\documentclass[11pt]{article}
\usepackage{amsmath,amssymb,latexsym,mathdots}
\usepackage[all,cmtip]{xy}
\usepackage{color}

\def\Z{{\mathbb Z}}

\def\SL{{\rm SL}}
\def\GL{{\rm GL}}

\def\OO{{\mathcal O}}
\def\SO{{\rm SO}}
\def\PSO{{\rm PSO}}

\def\s{{\rm stab.}}

\def\irr{{\rm irr}}

\def\dist{{\rm dist}}

\def\red{{\rm red}}
\def\Vol{{\rm Vol}}
\def\R{{\mathbb R}}

\def\F{{\mathbb F}}
\def\FF{{\mathcal F}}

\def\Q{{\mathbb Q}}

\def\C{{\mathcal C}}
\def\U{{\mathcal W}}
\def\W{{\mathcal W}}
\def\V{{\mathcal V}}

\def\Vmn{{V_n}}
\def\Z{{\mathbb Z}}

\def\F{{\mathbb F}}
\def\Q{{\mathbb Q}}
\def\C{{\mathbb C}}

\def\GG{{\mathcal G}}

\def\LL{{\mathcal L}}

\def\w{{\rm {(2)}}}
\def\s{{\rm {(1)}}}

\def\max{{\rm max}}
\def\Res{{\textrm{Res}}}
\def\sub{{\textrm{top}}}

\newtheorem{theorem}{Theorem}[section]
\newtheorem{corollary}[theorem]{Corollary}
\newtheorem{lemma}[theorem]{Lemma}
\newtheorem{proposition}[theorem]{Proposition}
\newtheorem{remark}[theorem]{Remark}
\newenvironment{proof}{\noindent {\bf Proof:}}{$\Box$ \vspace{2 ex}}

\title{Squarefree values of polynomial discriminants I}

\author{Manjul Bhargava, Arul Shankar, and Xiaoheng Wang}

\begin{document}

\maketitle

\begin{abstract}
We determine the density of monic integer polynomials of given degree
$n>1$ that have squarefree discriminant; in particular, we prove for the
first time that the lower density of such polynomials is positive.
Similarly, we prove that the density of monic integer polynomials
$f(x)$, such that $f(x)$ is irreducible and $\Z[x]/(f(x))$ is the ring
of integers in its fraction field, is positive, and is in fact given
by~$\zeta(2)^{-1}$.

It also follows from our methods that there are $\gg X^{1/2+1/n}$
monogenic number fields of degree~$n$ having associated Galois
group~$S_n$ and absolute discriminant less than $X$, and we conjecture
that the exponent in this lower bound is optimal.
\end{abstract}

\section{Introduction}

The pupose of this paper is to determine the density of monic integer
polynomials of given degree whose discriminant is squarefree.  For
polynomials $f(x)=x^n+a_1x^{n-1}+\cdots+a_n$, the term $(-1)^ia_i$
represents the sum of the $i$-fold products of the roots of $f$. It is
thus natural to order monic polynomials
$f(x)=x^n+a_1x^{n-1}+\cdots+a_n$ by the height
$H(f):=\max\{|a_i|^{1/i}\}$ (see, e.g., \cite{BG2}, \cite{PS2},
\cite{SW}).  We determine the density of monic integer polynomials
of degree $n$ having squarefree discriminant with respect to the ordering by this
height, and show that the density is positive.  The existence of infinitely many monic integer
polynomials of each degree having squarefree discriminant was
first demonstrated by Kedlaya~\cite{Kedlaya}.
However, it has not previously been known whether the density exists or even that the lower density
is positive.

To state the theorem, define the constants $\lambda_n(p)$ by
\begin{equation}\label{jos}
\lambda_n(p)=\left\{
\begin{array}{cl}
1 & \mbox{if $n =1$,}\\[.0875in]
1-\displaystyle\frac1{p^2} & \mbox{if $n= 2$,}\\[.1875in]
1-\displaystyle \frac{3p^{n-1}-p^{n-2}+(-1)^n(p-1)^2}{p^n(p+1)} & \mbox{if $n\geq 3$}
\end{array}\right.
\end{equation}
for $p\neq 2$; also, let $\lambda_1(2)=1$ and $\lambda_n(2)=1/2$ for
$n\geq2$.  Then a result of Yamamura~\cite[Proposition~3]{Yamamura2}
states
that $\lambda_n(p)$ is the density of monic polynomials of degree $n$ over~$\Z_p$
having discriminant indivisible by~$p^2$.
Let~$\lambda_n:=\prod_p\lambda_n(p)$, where the product is over all
primes $p$.  We~prove:

\begin{theorem}\label{polydisc2}
Let $n\geq1$ be an integer.  Then when
monic integer polynomials $f(x)=x^n+a_1x^{n-1}+\cdots+a_n$ of
degree~$n$ are ordered by $H(f):=
\max\{|a_1|,|a_2|^{1/2},\ldots,|a_n|^{1/n}\}$, the density having
squarefree discriminant $\Delta(f)$ exists and is equal to $\lambda_n>0$.
\end{theorem}
Our method of proof implies that the theorem remains true even if we
restrict only to those polynomials of a given degree $n$ having a
given number of real roots.

It is easy to see from the definition of the $\lambda_n(p)$ that the
$\lambda_n$ rapidly approach a limit $\lambda$ as $n\to\infty$,
namely,
\begin{equation}
\lambda=\lim_{n\to\infty} \lambda_n = \frac12 \cdot \prod_{p\geq 3}
\left(1-\displaystyle\frac{3p-1}{p^2(p+1)}\right)
\approx 30.7056\%.
\end{equation}
Therefore, as the degree tends to infinity, the
probability that a random monic integer polynomial has squarefree
discriminant tends to $\lambda\approx 30.7056\%$.

In algebraic number theory, one often considers number fields that are
defined as a quotient ring $K_f:=\Q[x]/(f(x))$ for some irreducible
integer polynomial $f(x)$. The question naturally arises as to whether
$R_f:=\Z[x]/(f(x))$ gives the ring of integers of $K_f$.  Our second
main theorem states that this is in fact the case for {most}
polynomials $f(x)$. We prove:

\begin{theorem}\label{polydiscmax2}
  Let $n > 1$ be an integer.  Then when monic integer polynomials
  $f(x)=x^n+a_1x^{n-1}+\cdots+a_n$ of degree~$n$ are ordered by
  $H(f)$, the density of polynomials $f$ such that $\Z[x]/(f(x))$ is
  the ring of integers in its fraction field is
  $\prod_p(1-1/p^2)=\zeta(2)^{-1}$.
\end{theorem}
Note that $\zeta(2)^{-1}\approx\, 60.7927\%$.  Since a density of
100\% of monic integer polynomials are irreducible (and indeed have
associated Galois group $S_n$) by Hilbert's irreducibility theorem, it
follows that $\approx 60.7927\%$ of monic integer polynomials $f$ of
any given degree $n>1$ have the property that $f$ is irreducible and
$\Z[x]/(f(x))$ is the maximal order in its fraction field. The
quantity
\begin{equation}\label{eqrho}
  \rho_n(p):=1-\frac{1}{p^2}
\end{equation}
represents the density of monic integer polynomials
of degree $n>1$ over $\Z_p$ such that $\Z_p[x]/(f(x))$ is the maximal
order in $\Q_p[x]/(f(x))$.  The determination of this beautiful
$p$-adic density, and its independence of $n$, is due to Hendrik
Lenstra (see~\cite[Proposition~3.5]{ABZ}).  Theorem~\ref{polydiscmax2}
again holds even if we restrict to polynomials of degree $n$ having a
fixed number of real roots.

If the discriminant of an order in a number field is squarefree, then
that order must be maximal.  Thus the irreducible polynomials counted
in Theorem~\ref{polydisc2} are a subset of those counted in
Theorem~\ref{polydiscmax2}.  The additional usefulness of
Theorem~\ref{polydisc2} in some arithmetic applications is that if
$f(x)$ is a monic irreducible integer polynomial of degree $n$ with
squarefree discriminant, then not only is $\Z[x]/(f(x))$ maximal in
the number field $\Q[x]/(f(x))$ but the associated Galois group is
necessarily the symmetric group $S_n$ (see, e.g., \cite{Scholz},
\cite{Yamamura}, \cite{Nakagawa}, \cite{Kondo} for further details and
applications).

We prove both Theorems~\ref{polydisc2} and \ref{polydiscmax2} with
power-saving error terms. More precisely, let $\Vmn(\Z)$ denote
the subset of $\Z[x]$ consisting of all monic integer polynomials of
degree $n$.  Then it is easy to see that
\begin{equation*}
{\#\{f\in \Vmn(\Z):  H(f)<X \}} = 2^nX^{\frac{\scriptstyle n(n+1)}{\scriptstyle2}} + O(X^{\frac{\scriptstyle n(n+1)}{\scriptstyle2}-{ 1}}).
\end{equation*}
We prove
\begin{equation}\label{errorterms}
{\begin{array}{ccl}
\displaystyle
{\#\{f\in \Vmn(\Z) : H(f)<X \mbox{ and $\Delta(f)$ squarefree}\}}
&\!\!=\!\!&  \lambda_n\cdot 2^n{X^\frac{\scriptstyle n(n+1)}{\scriptstyle2}} + O_\varepsilon(X^{\frac{\scriptstyle n(n+1)}{\scriptstyle2}-{\textstyle \frac15}+\varepsilon});\\[.15in]
\displaystyle
{\#\{f\in \Vmn(\Z) : H(f)<X \mbox{ and $\Z[x]/(f(x))$ maximal}\}}&\!\!=\!\!&  {\displaystyle\frac{6}{\pi^2}}\cdot 2^nX^{\frac{\scriptstyle n(n+1)}{\scriptstyle2}} + O_\varepsilon(X^{\frac{\scriptstyle n(n+1)}{\scriptstyle2}-{\textstyle \frac15}+\varepsilon})
\end{array}}
\end{equation}
for $n>1$.

These asymptotics imply Theorems~\ref{polydisc2} and
\ref{polydiscmax2}.  Since it is known that the number of reducible
monic polynomials of a given degree~$n$ is of a strictly smaller order
of magnitude than the error terms above (see Proposition
\ref{propredboundall}), it does not matter whether we require $f$ to
be irreducible in the above asymptotic formulae.

Recall that a number field $K$ is called {\it monogenic} if its ring
of integers is generated over~$\Z$ by one element, i.e., if
$\Z[\theta]$ gives the maximal order of $K$ for some $\theta\in K$.
As a further application of our methods, we obtain the following
corollary to Theorem~\ref{polydisc2}:

\begin{corollary}\label{monogenic}
Let $n>1$.  The number of isomorphism classes of number fields
of degree~$n$ and absolute discriminant less than $X$ that are
monogenic and have associated Galois group $S_n$ is $\gg X^{1/2+1/n}$.
\end{corollary}
We note that our lower bound for the number of monogenic $S_n$-number
fields of degree $n$ improves slightly the best-known lower bounds for
the number of $S_n$-number fields of degree $n$, due to Ellenberg and
Venkatesh~\cite[Theorem 1.1]{EV}, by simply forgetting the
monogenicity condition in Corollary~\ref{monogenic}.  We conjecture
that the exponent in our lower bound in Corollary~\ref{monogenic} for
monogenic number fields of degree $n$ is optimal.

As is illustrated by Corollary~\ref{monogenic},
Theorems~\ref{polydisc2} and \ref{polydiscmax2} give a powerful method
to produce number fields of a given degree having given properties or
invariants.  We give one further example of interest.  Given a number
field $K$ of degree $n$ with $r$ real embeddings $\xi_1,\dots,\xi_r$
and $s$ complex conjugate pairs of complex embeddings
$\xi_{r+1},\bar\xi_{r+1},\ldots,\xi_{r+s},\bar\xi_{r+s}$, the ring of
integers $\mathcal O_K$ may naturally be viewed as a lattice in $\R^n$
via the map $x\mapsto (\xi_1(x),\ldots,\xi_{r+s}(x))\in
\R^r\times\C^s\cong \R^n$.  We may thus ask about the length of the
shortest vector in this lattice generating $K$.

In their final remark~\cite[Remark~3.3]{EV}, Ellenberg and Venkatesh
conjecture that the number of number fields $K$ of degree $n$ whose
shortest vector in $\OO_K$ generating $K$ is of length less than~$Y$
is $\,\asymp Y^{(n-1)(n+2)/2}$. They prove an upper bound of this
order of magnitude.  We use Theorem~\ref{polydiscmax2} to prove also a
lower bound of this size, thereby proving their conjecture:

\begin{corollary}\label{shortvector}
Let $n>1$.  The number of isomorphism classes of number fields $K$ of
degree~$n$ whose shortest vector in
$\OO_K$ generating $K$ has length less than $Y$ is $\,\asymp$
$Y^{(n-1)(n+2)/2}$.  The same is true if we further impose the
condition that the Galois group of the normal closure of $K$ is $S_n$.
\end{corollary}


Finally, we remark that our methods allow the analogues of all of the
above results to be proven with any finite set of local conditions
imposed at finitely many places (including at infinity); the orders of
magnitudes in these theorems are then seen to remain the same---with
different (but easily computable in the cases of
Theorems~\ref{polydisc2} and \ref{polydiscmax2}) positive
constants---provided that no local conditions are imposed that force
the set being counted to be empty (i.e., no local conditions are
imposed at~$p$ in Theorem~\ref{polydisc2} that force $p^2$ to divide
the discriminant, no local conditions are imposed at~$p$ in
Theorem~\ref{polydiscmax2} that cause $\Z_p[x]/(f(x))$ to be
non-maximal over~$\Z_p$, and no local conditions are imposed at $p$ in
Corollary~\ref{monogenic} that cause such number fields to be
non-monogenic locally). In fact, we can even impose certain infinite
sets of local conditions (see Theorem \ref{thmaingencong}).

\vspace{.1in} We now briefly describe our methods.  It is easily seen
that the desired densities in Theorems~\ref{polydisc2} and
\ref{polydiscmax2}, if they exist, must be bounded above by the Euler
products $\prod_p \lambda_n(p)$ and $\prod_p (1-1/p^2)$,
respectively. The difficulty is to show that these Euler products are
also the correct lower bounds.  As is standard in sieve theory, to
demonstrate the lower bound, a ``tail estimate'' is required to show
that not too many discriminants of polynomials $f$ are divisible by
$p^2$ when $p$ is large relative to the discriminant $\Delta(f)$ of
$f$ (here, large means larger than $\Delta(f)^{1/(n-1)}$, say).

For any prime $p$, and a monic integer polynomial $f$ of degree $n$
such that $p^2\mid \Delta(f)$, we say that $p^2$ {\it strongly
  divides} $\Delta(f)$ if $p^2\mid \Delta(f + pg)$ for any integer
polynomial $g$ of degree $n$; otherwise, we say that $p^2$ {\it weakly
  divides} $\Delta(f)$. Then $p^2$ strongly divides $\Delta(f)$ if and
only if $f$ modulo $p$ has at least two distinct multiple roots in
$\bar{\F}_p$, or has a root in $\F_p$ of multiplicity at least 3; and
$p^2$ weakly divides $\Delta(f)$ if $p^2\mid \Delta(f)$ but $f$ modulo
$p$ has only one multiple root in $\F_p$ and this root is a simple
double root.

For any squarefree positive integer $m$, let $\U_m^\s$
(resp.\ $\U_m^\w$) denote the set of monic integer polynomials in
$\Vmn(\Z)$ whose discriminant is strongly divisible (resp.\ weakly
divisible) by $p^2$ for every prime factor $p$ of $m$.  Then we prove
tail estimates for $\U_m^\s$ and $\U_m^\w$ separately, as follows.

\begin{theorem}\label{thm:mainestimate}
For any positive real number $M$ and any $\epsilon>0$, we have

\vspace{-5pt}\begin{eqnarray*}
\label{eq:equs}
{\rm (a)}\quad \#\bigcup_{\substack{m>M\\ m\;\mathrm{ squarefree}
 }}\{f\in\U_m^\s:H(f)<X\}&=&
O_\epsilon(X^{n(n+1)/2+\epsilon}/M)+O(X^{n(n+1)/2-1});\\[.075in]
\label{equ1}
{\rm (b)}\quad
\#\bigcup_{\substack{m>M\\
m\;\mathrm{ squarefree}
}}\{f\in\U_m^\w:H(f)<X\}&=&
O_\epsilon(X^{n(n+1)/2+\epsilon}/M)+O_\epsilon(X^{n(n+1)/2-1/5+\epsilon}),
\end{eqnarray*}
where the implied constants are independent of $M$ and $X$.
\end{theorem}
\noindent To prove our main theorems, we will use Theorem
\ref{thm:mainestimate} with $M=X^{1/2}$.

The power savings in the error terms above also have applications towards
determining the distributions of low-lying zeros in families of
Dedekind zeta functions of monogenic degree-$n$ fields; see \cite[\S5.2]{SST}.

We prove the estimate in the strongly divisible case (a) of
Theorem~\ref{thm:mainestimate} by geometric techniques, namely, a
quantitative version of the Ekedahl sieve (\cite{Ek}, \cite[Theorem
  3.3]{geosieve}). While the proof of \cite[Theorem~3.3]{geosieve}
uses homogeneous heights, and considers the union over all primes
$p>M$, the same proof also applies in our case of weighted homogeneous
heights, and a union over all squarefree $m>M$. Since the last
coefficient $a_n$ is in a larger range than the other coefficients, we
in fact obtain a smaller error term than in \cite[Theorem~3.3]{geosieve}.

The estimate in the weakly divisible case (b) of
Theorem~\ref{thm:mainestimate} is considerably more difficult. Our
main idea is to embed polynomials $f$, whose discriminant is {\it weakly}
divisible by $p^2$, into a larger space that has more symmetry, such
that the invariants under this symmetry are given exactly by the
coefficients of $f$; moreover, we arrange for the image of $f$ in the
bigger space to have discriminant {\it strongly} divisible by $p^2$.  We
then count in the bigger space.

More precisely, we make use of the representation of $G=\SO_n$ on the
space $W=W_n$ of symmetric $n\times n$ matrices, as studied in
\cite{BG2,SW}. We fix $A_0$ to be the $n\times n$ symmetric matrix
with $1$'s on the anti-diagonal and $0$'s elsewhere. The group
$G=\SO(A_0)$ acts on $W$ via the action $g\cdot B=gBg^t$ for $g\in G$
and $B\in W$.  Define the {\it invariant polynomial} of an element
$B\in W$ by $$f_B(x) = (-1)^{n(n-1)/2}\det(A_0x - B).$$ Then $f_B$ is
a monic polynomial of degree~$n$. It is known (see \cite[\S 4]{AITBG})
that the ring of polynomial invariants for the action of $G$ on $W$ is
freely generated by the coefficients of the invariant
polynomial. Define the {\it discriminant} $\Delta(B)$ and {\it height}
$H(B)$ of an element $B\in W$ by $\Delta(B)=\Delta(f_B)$ and
$H(B)=H(f_B)$.  This representation of $G$ on $W$ was used in
\cite{BG2,SW} to study 2-descent on the hyperelliptic curves
$C:y^2=f_B(x)$.

A key step of our proof of Theorem~\ref{thm:mainestimate}(b) is the
construction, for every positive squarefree integer $m$, of a map
\begin{equation*}
\sigma_m:\U_m^\w\to \frac14W(\Z),
\end{equation*}
such that $f_{\sigma_m(f)}=f$ for every $f\in \U_m^\w$; here
$\frac14W(\Z)\subset W(\Q)$ is the lattice of elements $B$ whose
coefficients have denominators dividing $4$.  In our construction, the
image of $\sigma_m$ in fact lies in a special subspace $W_0$ of $W$;
namely, if $n=2g+1$ is odd, then $W_0$ consists of symmetric matrices
$B\in W$ whose top left $g\times g$ block is 0, and if $n=2g+2$ is
even, then $W_0$ consists of symmetric matrices $B\in W$ whose top
left $g\times (g+1)$ block is 0. We associate to any element of $W_0$
a further polynomial invariant which we call the $Q$-{\it invariant}
(which is a relative invariant for the subgroup of $\SO(A_0)$ that
fixes $W_0$).
Next, we show that for elements $B$ in the image of $\sigma_m$, we
have $|Q(B)|=m$.
Finally, {even though the discriminant polynomial of $f\in\U_m^\w$ is
  {\it weakly} divisible by $p^2$, the discriminant polynomial of its
  image $\sigma_m(f)$, when viewed as a polynomial on $W_0\cap
  \frac14W(\Z)$, is {\it strongly} divisible by $p^2$.}  This is the
key point of our construction.

To obtain Theorem~\ref{thm:mainestimate}(b), it thus suffices to
estimate the number of $G(\Z)$-equivalence classes of elements $B\in
W_0\cap \frac14W(\Z)$ of height less than $X$ having $Q$-invariant
larger than $M$.  This can be reduced to a geometry-of-numbers
argument in the spirit of \cite{BG2,SW}, although the current count is
more subtle in that we are counting certain elements in a cuspidal
region of a fundamental domain for the action of $G(\Z)$ on $W(\R)$.
The $G(\Q)$-orbits of elements $B\in W_0\cap W(\Q)$ are called {\it
  distinguished orbits} in \cite{BG2,SW}, as they correspond to the
identity 2-Selmer elements of the Jacobians of the corresponding
hyperelliptic curves $y^2=f_B(x)$ over $\Q$; these were not counted
separately by the geometry-of-numbers methods of \cite{BG2,SW}, as
these elements lie deeper in the cusps of the fundamental domains.  We
develop a method to count those elements in the cusp having bounded height and $Q$-invariant larger than $M$, following
the arguments of \cite{BG2,SW} while using the invariance and
algebraic properties of the $Q$-invariant polynomial.  This yields
Theorem~\ref{thm:mainestimate}(b), which then allows us to carry out
the sieves required to obtain Theorems~\ref{polydisc2} and
\ref{polydiscmax2}.

Corollary~\ref{monogenic} can be deduced from
Theorem~\ref{polydisc2} roughly as follows. Let $g\in \Vmn(\R)$ be
a monic real polynomial of degree~$n$ and nonzero discriminant having
$r$ real roots and $2s$ complex roots.  Then $\R[x]/(g(x))$ is
isomorphic to $\R^n\cong \R^r\times \C^s$ via its real and complex
embeddings.  Let $\theta$ denote the image of $x$ in $\R[x]/(g(x))$
and let $R_g$ denote the lattice formed by taking the $\Z$-span of
$1,\theta,\ldots,\theta^{n-1}$. Suppose further that there exist
monic integer polynomials $h_i$ of degree $i$ for $i=1,\ldots,n-1$
such that $1,h_1(\theta),h_2(\theta),\ldots,h_{n-1}(\theta)$ is the
unique Minkowski-reduced basis of $R_g$; we say that the polynomial
$g(x)$ is {\it strongly quasi-reduced} in this case. Note that if $g$
is an integer polynomial, then the lattice $R_g$ is simply the image
of the ring $\Z[x]/(g(x))\subset \R[x]/(g(x))$ in $\R^n$ via its
archimedean embeddings.

  When ordered by their heights, we prove that 100\% of monic integer
  polynomials $g(x)$ are strongly quasi-reduced. We furthermore prove
  that two distinct strongly quasi-reduced integer polynomials $g(x)$
  and $g^\ast(x)$ of degree~$n$ with vanishing $x^{n-1}$-term
  necessarily yield non-isomorphic rings $R_g$ and $R_{g^\ast}$.  The
  proof of the positive density result of Theorem~\ref{polydisc2} then
  produces $\gg X^{1/2+1/n}$ strongly quasi-reduced monic integer
  polynomials $g(x)$ of degree~$n$ having vanishing $x^{n-1}$-term,
  squarefree discriminant, and height less than $X^{1/(n(n-1))}$.
  These therefore correspond to $\gg X^{1/2+1/n}$ non-isomorphic
  monogenic rings of integers in $S_n$-number fields of degree $n$
  having absolute discriminant less than~$X$, and
  Corollary~\ref{monogenic} follows.

A similar argument proves Corollary~\ref{shortvector}. Suppose $f(x)$
is a strongly quasi-reduced irreducible monic integer polynomial of
degree $n$ with squarefree discriminant $\Delta(f)$.  Elementary
estimates show that if $H(f)<Y$, then $\|\theta\|\ll Y$, and so the
shortest vector in the ring of integers generating the field also has
length bounded by $O(Y)$. The above-mentioned result on the number of
strongly quasi-reduced irreducible monic integer polynomial of degree
$n$ with squarefree discriminant, vanishing $x^{n-1}$-coefficient, and
height bounded by $Y$ then gives the desired lower bound of
$\gg Y^{(n-1)(n+2)/2}.$ We give full proofs of Corollaries \ref{monogenic} and
\ref{shortvector} in Section~5.

In a subsequent paper \cite{SqSieve2} (Part II), we prove the
corresponding results for non-monic integer polynomials of degree $n$
ordered by the maximum of the absolute values of the coefficients.
Namely, we determine the density of such polynomials 
having squarefree discriminant and the density corresponding to maximal orders. The
treatment of non-monic integer polynomials in Part II builds on the
ideas here, but involves a number of new ideas due to the fact that
there exist non-monic integer polynomials $f(x)$ of degree $n$ for
which there are no symmetric integer matrices $A$ and $B$ such that
$f(x) = \det(Ax-B)$!  (See \cite{BGWhyper}.)  This
complication requires additional methods to adapt the proof to the
non-monic case.  The non-monic case has a number of applications as
well, including new results towards counting number fields and also to
the resolution of a conjecture of Poonen \cite{PBertini} on an
arithmetic Bertini theorem for the projective line. The main
theorems of this paper are the analogous results for the affine
line. The results in \cite{PBertini}, building on work of
Granville~\cite{Gabc}, imply versions of Theorems \ref{polydisc2} and
\ref{polydiscmax2} conditional on the ABC Conjecture.

\vspace{.1in} 
This paper is organized as follows. In Sections~\ref{sec:monicodd} and~\ref{sec:moniceven}, we begin by
collecting some algebraic facts about the representation $2\otimes
g\otimes(g+1)$ of $\SL_2\times\GL_g\times\GL_{g+1}$ and we define the
$Q$-invariant, which is a relative polynomial invariant for this
action.  We
then apply geometry-of-numbers techniques as described above to prove
the critical estimates of Theorem~\ref{thm:mainestimate}, handling the cases of $n$~odd and $n$~even separately. In
Section~\ref{sec:sieve}, we then show how our main theorems, Theorems~\ref{polydisc2} and \ref{polydiscmax2}, can be deduced from
Theorem~\ref{thm:mainestimate}.  Finally, in Section~\ref{latticearg},
we prove Corollary~\ref{monogenic} on the number of monogenic
$S_n$-number fields of degree~$n$ having bounded absolute
discriminant, as well as Corollary~\ref{shortvector} on the number of
rings of integers in number fields of degree $n$ whose shortest vector
generating the number field is of bounded length.

\section{A uniformity estimate for odd degree monic polynomials}\label{sec:monicodd}

In this section, we prove the estimate of
Theorem~\ref{thm:mainestimate}(b) when $n=2g+1$ is odd, for any $g\geq
1$.

\subsection{Invariant theory for the fundamental representation:
  $\SO_n$ on the space $W$ of symmetric $n\times n$ matrices}

Let $A_0$ denote the $n\times n$ symmetric matrix with $1$'s on the
anti-diagonal and $0$'s elsewhere. The group $G=\SO(A_0)$ acts on $W$
via the action
\begin{equation*}
\gamma\cdot B=\gamma B\gamma^t.
\end{equation*}

We recall some of the arithmetic invariant theory for the
representation $W$ of $n\times n$ symmetric matrices of the split
orthogonal group $G$; see \cite{BG2} for more details.  The ring of
polynomial invariants for the action of $G(\C)$ on $W(\C)$ is freely
generated by the coefficients of the {\it invariant polynomial
  $f_B(x)$ of $B$}, defined by $$f_B(x):=(-1)^{g}\det(A_0x-B)$$ (see
\cite[\S4]{AITBG}). We define the {\it discriminant} $\Delta$ on $W$
by $\Delta(B)=\Delta(f_B)$, and the $G(\R)$-invariant {\it height} of
elements in $W(\R)$ by $H(B)=H(f_B).$

Let $k$ be any field of characteristic not $2$.
For a monic polynomial $f(x)\in k[x]$ of degree~$n$ such that
$\Delta(f)\neq0$, let $C_f$ denote the smooth hyperelliptic curve
$y^2=f(x)$ of genus $g$ and let $J_f$ denote the Jacobian of
$C_f$. Then $C_f$ has a rational Weierstrass point at infinity.
The stabilizer of an element $B\in W(k)$ with invariant polynomial
$f(x)$ is naturally isomorphic to $J_f[2](k)$
by~\cite[Proposition~5.1]{BG2}, and hence has cardinality at most
$\#J_f[2](\bar k)=2^{2g}$, where $\bar k$ denotes a separable closure
of $k$.

We say that an element (or the $G(k)$-orbit of an element) $B\in W(k)$ with $\Delta(B)\neq 0$
is {\it $k$-distinguished} if there exists a $g$-dimensional subspace
defined over $k$ that is isotropic with respect to both $A_0$ and~$B$. If $B$ is $k$-distinguished, then the set of these $g$-dimensional
subspaces over $k$ is in bijection with $J_f[2](k)$
by~\cite[Proposition~4.1]{BG2}, and so it too has cardinality at most
$2^{2g}$.

In fact, it is known (see \cite[Proposition~5.1]{BG2}) that the
elements of $J_f[2](k)$ are in natural bijection with the even-degree
factors of $f$ defined over $k$.  (Note that the number of even-degree
factors of $f$ over $\bar k$ is indeed $2^{2g}$.) In particular, if
$f$ is irreducible over $k$, then the group $J_f[2](k)$ is trivial.

Now let $W_0$ be the subspace of $W$ consisting of matrices whose top left
$g\times g$ block is zero. Then elements $B$ in $W_0(k)$ with nonzero
discriminant are all evidently $k$-distinguished since the $g$-dimensional subspace
$Y_g$ spanned by the first $g$ basis vectors is isotropic with respect
to both $A_0$ and $B$.
Let $G_0$ denote the subgroup of $G$
consisting of elements $\gamma$ such that $\gamma^t$ preserves
$Y_g$. Then $G_0$ acts on $W_0$.

An element $\gamma\in G_0$ has the block matrix form
\begin{equation}\label{eq:G_0}
\gamma=\Bigl(\begin{array}{cc}\gamma_1 & 0\\ \delta & \gamma_2
\end{array}\Bigr)\in\Bigl(\begin{array}{cc}M_{g\times g} & 0\\ M_{(g+1)\times g} & M_{(g+1)\times (g+1)}
\end{array}\Bigr),
\end{equation}
and so $\gamma\in G_0$ transforms the top right $g\times (g+1)$ block
of an element $B\in W_0$ as follows:
$$(\gamma\cdot B)^\sub = \gamma_1B^\sub\gamma_2^t,$$ where we use the
superscript ``top'' to denote the top right $g\times (g+1)$ block of
any given element in $W_0$.  It will be convenient for us to view
$(A_0^\sub,B^\sub)$ as an element of the representation $V_g=2\otimes
g\otimes (g+1)$ of the group $H_g:=\SL_2\times\GL_g\times\GL_{g+1}$.
We have a map $\theta:G_0\to H_g$ sending $\gamma$ expressed in \eqref{eq:G_0} to
$(1,\gamma_1,\gamma_2)$. Then we have $$(A_0^\sub,(\gamma\cdot
B)^\sub)=\theta(\gamma)\cdot(A_0^\sub,B^\sub)$$ for $\gamma\in G_0$ and
$B\in W_0$.


Next, we construct a relative polynomial invariant for the action of $H_g$ on
$V_g$ as follows.  We write any $2\times g\times (g+1)$ matrix $v$ in
$V_g$ as a pair $(A,B)$ of $g\times(g+1)$ matrices.  Let $M_v(x,y)$
denote the vector of $g\times g$ minors of $Ax-By$, where $x$ and $y$
are indeterminates; in other words, the $i$-th coordinate of the
vector $M_v(x,y)$ is given by $(-1)^{i-1}$ times the determinant of
the matrix obtained by removing the $i$-th column of $Ax-By$.  Then
$M_v(x,y)$ is a vector of length $g+1$ consisting of binary forms of
degree $g$ in $x$ and~$y$, each of which has $g+1$ coefficients.
Taking the determinant of the resulting $(g+1)\times(g+1)$ matrix of
coefficients of these $g+1$ binary forms in $M_v(x,y)$ then yields a
polynomial $Q=Q(v)$ in the coordinates of $V_g$, which is a relative
invariant for the action of $H_g$. Explicitly, we have
\begin{equation}\label{eq:relQ}
Q((\gamma_0,\gamma_1,\gamma_2)\cdot v)=\det(\gamma_1)^{g+1}\det(\gamma_2)^g Q(v).
\end{equation}
Indeed, it follows from the definition that the $Q$-invariant is
invariant under the action of the subgroup
$\SL_2\times\SL_g\times\SL_{g+1}$. (Alternatively, the group
$\SL_2\times\SL_g\times\SL_{g+1}$ has no nontrivial characters.)
Moreover, we may work over the algebraic closure and write $\gamma_1$
and $\gamma_2$ as products of scalar matrices and matrices with
determinant 1.  Finally one can easily check \eqref{eq:relQ} when
$\gamma_1$ and $\gamma_2$ are scalar matrices.

We then define the $Q$-{\it invariant} of $B\in W_0$ to be the
$Q$-invariant of $(A_0^\sub,B^\sub)$:
\begin{equation}\label{eqQB}
Q(B):=Q(A_0^\sub,B^\sub).
\end{equation}
Then the $Q$-invariant is also a relative invariant for the action of $G_0$ on $W_0$, since for any
$\gamma\in G_0$ expressed in the form \eqref{eq:G_0}, we have
\begin{equation}\label{eq:weightG_0}
Q(\gamma\cdot B) = \det(\gamma_1)Q(B).
\end{equation}
In fact, we may extend the definition of the $Q$-invariant to an even
larger subset of $W(\Q)$ than $W_0(\Q)$. We have the following
proposition.

\begin{proposition}\label{prop:extendQ}
Let $B\in W_0(\Q)$ be an element whose invariant polynomial $f(x)$ is
irreducible over~$\Q$. Then for every $B'\in W_0(\Q)$ such that $B'$ is $G(\Z)$-equivalent to $B$, we have
$Q(B')=\pm Q(B)$.
\end{proposition}
\begin{proof}
Suppose $B'=\gamma\cdot B$ with $\gamma\in G(\Z)$ and $B,B'\in
W_0(\Q)$.  Then $Y_g$ and $\gamma^t Y_g$ are both $g$-dimensional
subspaces over $\Q$ isotropic with respect to both $A_0$ and
$B$. Since $f$ is irreducible over $\Q$, we have that $J_f[2](\Q)$ is
trivial, and so these two subspaces must be the same.  We conclude
that $\gamma\in G_0(\Z)$, and thus $Q(\gamma\cdot B)=\pm Q(B)$
by~\eqref{eq:weightG_0}.
\end{proof}

We may thus define the $|Q|$-{\it invariant} for any element $B\in
W(\Q)$ that is $G(\Z)$-equivalent to some element $B'\in W_0(\Q)$ and
whose invariant polynomial is irreducible over $\Q$; indeed, we set
$|Q|(B):=|Q(B')|$. By Proposition~\ref{prop:extendQ}, this definition
of $|Q|(B)$ is independent of the choice of $B'$. Note that all such
elements $B\in W(\Q)$ are $\Q$-distinguished.

\subsection{Embedding $\U_m^\w$ into $\frac12W(\Z)$}\label{sembedodd}

We begin by describing those monic integer polynomials in $\Vmn(\Z)$
that lie in $\U_m^\w$, i.e., the monic integer polynomials that have
discriminant weakly divisible by $p^2$ for all $p\mid m$.

\begin{proposition}\label{prop:hiddenp2}
Let $m$ be a positive squarefree integer, and let $f$ be a monic
integer polynomial whose discriminant is weakly divisible by $p^2$ for all $p\mid m$.
Then there exists an integer $\ell$ such that $f(x+\ell )$ has
the form
\begin{equation}
f(x+\ell ) = x^n + b_1x^{n-1} + \cdots + b_{n-2}x^2 + b_{n-1}x +
b_n
\end{equation}
for some integers $b_1,\ldots,b_n$ where $m$ divides $b_{n-1}$ and $m^2$ divides $b_n$.
\end{proposition}
\begin{proof}
Since $m$ is squarefree, by the Chinese Remainder Theorem it suffices
to prove the assertion in the case that $m=p$ is prime. Since $p$
divides the discriminant of~$f$, the reduction of $f$ modulo~$p$ must
have a repeated factor $h(x)^e$ for some polynomial $h\in \F_p[x]$ and
some integer $e\geq2$. As the discriminant of $f$ is not strongly
divisible by $p^2$, we see that $h$ is linear and $e=2$. By replacing
$f(x)$ by $f(x+\ell )$ for some integer $\ell$, if necessary, we may
assume that the repeated factor is $x^2$, i.e., we may assume that
$f(x)$ has the form $$f(x) = x^n + b_1x^{n-1} + \cdots + b_{n-1}x +
b_n$$ for some integers $b_1,\ldots,b_n$ such that $p$ divides
$b_{n-1}$ and $b_n$. It remains now to show that $p^2$ divides~$b_n$.

Viewing the discriminant $\Delta(f)$ has a polynomial in $b_n$,
we write
$$\Delta(f) = b_n\Delta_1 + \Delta_2,$$ for polynomials $\Delta_1\in
\Z[b_1,\ldots,b_n]$ and $\Delta_2\in \Z[b_1,\ldots,b_{n-1}]$. Next we
set $b_n$ to $0$ and observe that $b_{n-1}^2$ divides
$\Delta_2$. Indeed, the discriminant of $f$ is equal to the resultant
$\Res(f(x),f'(x))$ of $f(x)$ and $f'(x)$.  When $b_n=0$, the only
nonzero entry in the last column of the matrix whose determinant
computes $\Res(f(x),f'(x))$ is $b_{n-1}$ appearing in the last
row. The corresponding minor after removing the last row and column
has two nonzero entries in its last column, both of which equal to
$b_{n-1}$. This implies that we can pull out another factor of $b_{n-1}$ from the determinant of this minor.
Hence there is a polynomial $\Delta_3\in
\Z[b_1,\ldots,b_{n-1}]$ such that
\begin{equation}\Delta(f) = b_n\Delta_1 + b_{n-1}^2\Delta_3.
\end{equation}
Since $p$ divides $b_{n-1}$ and $b_n$, we see that $p^2\mid \Delta(f)$
if and only if $p^2\mid b_n\Delta_1$. If $p^2$ does not divide $b_n$,
then $p$ divides $\Delta_1$, which implies that $p^2$ divides $\Delta(f)$
strongly, a contradiction. Therefore, we have that $p^2$ divides $b_n$.
\end{proof}

Proposition \ref{prop:hiddenp2} identifies the $p^2$ hidden in the
coefficients of a monic polynomial $f$ when $p^2$ weakly divides
$\Delta(f)$. We next construct a matrix in $\frac12 W_0(\Z)$ with
invariant polynomial $f$ and where the two $p$'s in $p^2$ are split
apart.  For any integers $m,c_1,\ldots,c_n$, consider the matrix
\begin{equation}\label{mat1}
 B_m(c_1,\ldots,c_n) =
 \left(\!\begin{array}{ccccccccc}&&&&&&&m&\!0\\[.1in]&&&&&&\iddots&\iddots& \\[.15in]&&&&&\,1\,&\,0\,&&\\[.125in] &&&&\,1\,&0&&&\\[.125in] &&&1&-c_1&\!\!-c_2/2\!\!&&& \\[.15in] &&1&\;\,0\;\,&\!\!-c_2/2\!\!&-c_3&\!\!-c_4/2\!\!&&\\[.045in]&\iddots&\;\,\,0\;\,\,&&&\!\!-c_4/2\!\!&-c_5&\ddots&\\[.025in] \;m\;&\,\,\,\iddots\,\,\,&&&&&\ddots&\ddots&\!\!\!-c_{n-1}/2\!\!\!
   \\[.105in] \,0\,&&&&&&&\!\!\!-c_{n-1}/2\!\!\!&-c_n \end{array}\right)
\end{equation}
in $\frac12 W_0(\Z)$. It follows from a direct computation that
$$f_{B_m(c_1,\ldots,c_n)}(x) = x^n + c_1x^{n-1} + \cdots + c_{n-2}x^2
+ mc_{n-1}x + m^2c_n.$$ We compute the $Q$-invariant of $B_m(c_1,\ldots,c_n)$. Let $v = (A_0^\sub, B_m(c_1,\ldots,c_n)^\sub)$ be the corresponding element in $V_g$. The vector $M_v(x,y)$ of length $g+1$ of minors of $A_0^\sub x-B_m(c_1,\ldots,c_n)^\sub y$ is
$$M_v(x,y) = \begin{pmatrix} x^g\\ x^{g-1}y\\ \vdots \\ xy^{g-1}\\ my^g\end{pmatrix}.$$
Computing the determinant of the $(g+1)\times(g+1)$ matrix of coefficients of $M_v(x,y)$ then gives
$$Q(B_m(c_1,\ldots,c_n)) = Q(v) = m.$$
For any integer $\ell$, we have
\begin{equation}\label{eq:ell}
f_{B_m(c_1,\ldots,c_n) + \ell A_0}(x) = (-1)^g\det(A_0(x-\ell)-B_m(c_1,\ldots,c_n)) = f_{B_m(c_1,\ldots,c_n)}(x - \ell),
\end{equation}
and by the $\SL_2$-invariance of the $Q$-invariant,
\begin{equation}\label{eq:ellQ}
Q(B_m(c_1,\ldots,c_n) + \ell A_0) = Q(B_m(c_1,\ldots,c_n)) = m.
\end{equation}

\begin{theorem}\label{keymap}
Let $m$ be a positive squarefree integer. There exists a map
$\sigma_m:\U_m^\w\to \frac12W_0(\Z)$ such that for every $f\in \U_m^\w$,
\begin{equation}\label{eq:liftQ}
f_{\sigma_m(f)}=f,\qquad Q(\sigma_m(f)) = m.
\end{equation}
\end{theorem}

\begin{proof}
Let $f$ be any element of $\U_m^\w$. By Proposition \ref{prop:hiddenp2}, there exists an integer $\ell$ and integers $c_1,\ldots,c_n$ such that
$$f(x+\ell) = x^n + c_1x^{n-1} + \cdots + c_{n-2}x^2 + mc_{n-1}x + m^2c_n.$$
We set $\sigma_m(f) = B_m(c_1,\ldots,c_n) + \ell A_0$. Then \eqref{eq:liftQ} follows from  \eqref{eq:ell} and \eqref{eq:ellQ}.
\end{proof}

We remark that the integer $\ell$ in the proof of Theorem \ref{keymap}
is unique modulo $m$ since it is the unique double root of $f(x)$
modulo every prime factor $p$ of $m$. Since the invariant polynomial
of $\sigma_m(f)$ recovers $f$, we see that no two elements in the
image of $\sigma_m$ are in the same $G(\Z)$-orbit (or even
$G(\C)$-orbit). Furthermore, by the definition of discriminant on $W$,
we have that $\Delta(\sigma_m(f))=\Delta(f)$ is divisible by $p^2$ for
every prime factor $p$ of $m$. If one varies the coefficients of
$\sigma_m(f)$ by multiples of $p$ while keeping the top left $g\times
g$ block $0$, one can show that $p^2$ still divides the discriminant
of the resulting matrix. In other words, the matrix $\sigma_m(f)$, as
an element of $\frac12 W_0(\Z)$, has discriminant \emph{strongly
  divisible} by $p^2$. It then seems natural to use the Ekedahl sieve
to handle this strongly divisible case. However, the Ekedahl sieve
does not apply when the polynomial is not squarefree as a polynomial
and as we will show in the sequel, as polynomials in the coordinates
of $W_0$, the discriminant $\Delta$ is divisible by $Q^2$. Instead, we
will count elements in $\frac12W_0(\Z)$ having large $Q$-invariant
using geometry-of-numbers techniques.

More precisely, let $\LL$ be the set of elements $v\in \frac12W(\Z)$
satisfying the following conditions: $v$ is $G(\Z)$-equivalent to some
element in $\frac12W_0(\Z)$ and the invariant polynomial of $v$ is
irreducible over $\Q$.  Then by the remark following
Proposition~\ref{prop:extendQ}, we may view $|Q|$ as a function also
on $\LL$.  Using $\U_m^{\w,\irr}$ to denote the set of irreducible
polynomials in $\U_m^\w$, we then have the following immediate
consequence of Theorem \ref{keymap}:

\begin{theorem}\label{keymaporbit}
Let $m$ be a positive squarefree integer. There exists an injective map
$$\bar{\sigma}_m:\U_m^{\w,\irr}\to G(\Z)\backslash\LL$$ such that
$f_{\bar{\sigma}_m(f)}=f$ for every $f\in \U_m^{\w,\irr}$. Moreover,
for every element $B$ in the $G(\Z)$-orbit of an element in the image
of $\bar{\sigma}_m$, we have $|Q|(B)=m$.
\end{theorem}

The number of reducible monic integer
polynomials having height less than $X$ is of a strictly smaller order
of magnitude than the total number of such polynomials (see, e.g.,
Proposition~\ref{propredboundall}).  Thus, for our purposes of proving
Theorem~\ref{thm:mainestimate}(b), it will suffice to count elements
in $\U_m^{\w,\irr}$ of height less than $X$ over all $m>M$, which by
Theorem~\ref{keymaporbit} we may do by counting these special
$G(\Z)$-orbits on $\LL\subset \frac12W(\Z)$ having height less than
$X$ and $|Q|$-invariant greater than $M$.  More precisely, let
$N(\LL;M;X)$ denote the number of $G(\Z)$-equivalence classes of
elements in $\LL$ whose $|Q|$-invariant is greater than $M$ and whose
height is less than $X$.  Then, by Theorem~\ref{keymaporbit}, to
obtain an upper bound for the left hand side in
Theorem~\ref{thm:mainestimate}(b), it suffices to obtain the same
upper bound for $N(\LL;M;X)$.

\subsection{Counting $G(\Z)$-orbits in $\frac12W(\Z)$}

The counting problem for the representation $W$ of $G$ is studied in
\cite{BG2}. In this section, we recall some of the set up and results
of \cite{BG2}.

We remark first that the representation studied in \cite{BG2} is the
subspace of $W$ consisting of symmetric matrices with anti-trace
$0$. Since this is a linear condition on the anti-diagonal entries,
each of which has weight $1$ when restricted to the action of the
maximal torus (see also \eqref{wbij}), the only difference in not
imposing this condition is an extra factor of $X$ in the count. We
will in fact make use of the anti-trace-$0$ version in Section~\ref{latticearg} in our
application to counting fields.

To count $G(\Z)$-orbits in a lattice $\frac12W(\Z)$ in $W(\R)$, one
begins by constructing fundamental domains for the action of $G(\Z)$
on the set of elements in $W(\R)$ with nonzero discriminant. A
fundamental domain $R$ for the action of $G(\R)$ on the set of
elements in $W(\R)$ with nonzero discriminant and height less than $1$
is obtained in \cite[\S9.1]{BG2}. The exact shape of $R$ is not
important. What is important is that $R$ is (absolutely) bounded. Next
a fundamental domain $\FF$ for the left multiplication action of
$G(\Z)$ on $G(\R)$ is written down in \cite[\S9.2]{BG2}. This is done
using the Iwasawa decomposition of $G(\R)$ as
$$
G(\R)=N(\R)TK,
$$ where $N$ is a unipotent group consisting of lower triangular matrices, $K$ is compact, and $T$ is the split torus of $G$ given by
\begin{equation*}
T=
\left\{\left(\begin{array}{ccccccc}
 t_1^{-1}&&&&&&\\
&\ddots &&&&&  \\
 && t_{g}^{-1} &&&&\\
&&& 1 &&&\\
 &&&& t_g &&\\
&&&&&\ddots &  \\
& &&&&& t_{1}
\end{array}\right):t_1,\ldots,t_g\in\R\right\}.
\end{equation*}
We may also make the following change of variables. For $1\leq
i\leq g-1$, set $s_i$ to be
$$
s_i=t_i/t_{i+1},
$$ and set $s_g=t_g$.  It follows that for $1\leq i\leq g$, we have
$t_i = s_is_{i+1}\cdots s_g.$ We denote an element of $T$ with
coordinates $t_i$ (resp.\ $s_i$) by $(t)$ (resp.\ $(s)$). A
fundamental set for the action of $G(\Z)$ on $G(\R)$ can then be taken
to be contained in a {\it Siegel set}, i.e., contained in $N'T'K$,
where $N'$ consists of elements in $N(\R)$ whose coefficients are
absolutely bounded and $T'\subset T$ consists of elements in $(s)\in
T$ with $s_i\geq c$ for some positive constant $c$.

For any $h\in G(\R)$, since $\FF h$ remains a fundamental domain for
the action of $G(\Z)$ on $G(\R)$, the set $(\FF h)\cdot (XR)$ (when
viewed as a multiset) is a finite cover of a fundamental domain for
the action of $G(\Z)$ on the elements in $W(\R)$ with nonzero
discriminant and height bounded by $X$. The degree of the cover
depends only on the size of stabilizer in $G(\R)$ and is thus
absolutely bounded by $2^{2g}$. The presence of these stabilizers is
in fact the reason we consider $\FF h\cdot XR$ as a multiset. Hence,
we have
\begin{equation}\label{eqoddfundv}
  N(\LL;M;X)\ll \#\{B\in ((\FF h)\cdot (XR))\cap\LL:|Q|(B)>M\}.
\end{equation}
Let $\GG_1$ be a compact left $K$-invariant set in $G(\R)$ which is the
closure of a nonempty open set.
Averaging \eqref{eqoddfundv} over $h\in \GG_1$ and exchanging the
order of integration as in \cite[\S10.1]{BG2}, we obtain
\begin{equation}
  N(\LL;M;X)\ll \int_{\gamma\in\FF} \#\{B\in((\gamma \GG_1)\cdot (XR))\cap\LL:|Q|(B)>M\}
  d\gamma,
\end{equation}
where the implied constant depends only on $\GG_1$ and $R$, and where $d\gamma$ is a Haar measure on $G(\R)$ given by
$$
d\gamma=dn\,\delta(s)d^\times s\,dk,
$$ where $dn$ is a Haar measure on the unipotent group $N(\R)$,
$dk$ is a Haar measure on the compact group~$K$, $d^\times s$ is
given by
$$
d^\times s:=\prod_{i=1}^g\frac{ds_i}{s_i},
$$
and
\begin{equation}\label{eqhaarodd}
\delta(s)=\prod_{k=1}^g s_k^{k^2-2kg};
\end{equation}
see \cite[(10.7)]{BG2}.

Since $s_i\geq c$ for every $i$, there exists a compact subset $N''$
of $N(\R)$ containing $(t)^{-1}N'\,(t)$ for all $t\in T'$. Since
$N''$, $K$, $\GG_1$ are compact and $R$ is bounded, the set $E=N''K\GG_1R$
is bounded. Then we have
\begin{equation}\label{eq:RE}
N(\LL;M;X)\ll \int_{s_i\gg 1} \#\{B\in((s)\cdot XE)\cap\LL:|Q|(B)>M\}\delta(s)d^\times s.
\end{equation}

We denote the coordinates on $W$ by $b_{ij}$, for $1\leq i\leq j\leq
n$. These coordinates are eigenvectors for the action of $T$ on $W^*$,
the dual of $W$. Denote the $T$-weight of a coordinate $\alpha$ on
$W$, or more generally a product $\alpha$ of powers of such
coordinates, by $w(\alpha)$. An elementary computation shows that
\begin{equation}\label{wbij}
w(b_{ij})=\left\{
\begin{array}{rcl}
t_i^{-1}t_j^{-1} &\mbox{ if }& i,j\leq g\\
t_i^{-1} &\mbox{ if }& i\leq g,\;j=g+1\\
t_i^{-1}t_{n-j+1} &\mbox{ if }& i\leq g,\; j>g+1\\
1 &\mbox{ if }& i=j=g+1\\
t_{n-j+1} &\mbox{ if }& i=g+1,\;j>g+1\\
t_{n-i+1}t_{n-j+1} &\mbox{ if }& i,j>g+1.
\end{array}
\right.
\end{equation}
Then the $(i,j)$-entry of any $B\in (s)\cdot XE$ is bounded by
$Xw(b_{ij})$, up to a multiplicative constant depending only on $\GG_1$
and $R$.

Let $W^\dist$ denote the subset of $W(\Q)$ consisting of
$\Q$-distinguished elements. Then $\LL$ is a subset of $W^\dist$. It
is shown in \cite[\S10.2]{BG2} that most of the lattice points in
$W^\dist$ lie inside the subspace $W_{00}$ consisting of symmetric
matrices $B$ whose $(i,j)$-entries are 0 whenever $i+j<n$. More
precisely, we have the following estimates from \cite[Propositions~10.5 and 10.7]{BG2}:

\begin{proposition}\label{prop:bg2}
We have
\begin{eqnarray}
  \int_{s_i\gg 1} \#\{B\in((s)\cdot XE)\cap({\textstyle\frac12} W(\Z)\setminus{\textstyle\frac12} W_{00}(\Z)):b_{11}=0\}\delta(s)d^\times s &=& O_\epsilon(X^{n(n+1)/2-1+\epsilon}) \\
  \label{eq:selbergpre}\int_{s_i\gg 1} \#\{B\in((s)\cdot XE)\cap\LL:b_{11}\neq 0\}\delta(s)d^\times s &=& o(X^{n(n+1)/2}).
\end{eqnarray}
\end{proposition}

Therefore, to prove Theorem~\ref{thm:mainestimate}(b) when $n$ is odd, it remains to obtain a power-saving improvement of \eqref{eq:selbergpre} and estimate
\begin{equation}\label{eq:Qinv}
\int_{s_i\gg 1} \#\{B\in ((s)\cdot XE)\cap \LL\cap
{\textstyle\frac12} W_{00}(\Z):|Q|(B)>M\}\delta(s)d^\times s.
\end{equation}

\subsection{Proof of Theorem~\ref{thm:mainestimate}(b)
  for odd $n$}\label{sgomodd}

We begin with a power-saving improvement of \eqref{eq:selbergpre}.

\begin{proposition}\label{propodd}
  We have
  \begin{equation*}
\displaystyle\int_{s_i\gg 1} \#\{B\in ((s)\cdot XE)\cap W^\dist \cap {\textstyle\frac12} W(\Z): b_{11}\neq 0\}\delta(s)d^\times s =O_\epsilon(X^{n(n+1)/2-1/5+\epsilon}).
  \end{equation*}
\end{proposition}
\begin{proof}
We note that the $p$-adic density of elements in $W(\Z_p)$ that are
$\Q_p$-distinguished is bounded uniformly away from $1$. In fact, this
density is bounded above by $1 - \frac{g}{2g+1}+O(1/p)$ by 
\cite[\S10.7]{BG2}.  Then an application of the Selberg sieve exactly
as in \cite{ShTs} yields the result.
\end{proof}

We now estimate \eqref{eq:Qinv}. The $Q$-invariant can also be
given a weight by viewing the torus $T$ as sitting inside $G_0$ and
using \eqref{eq:weightG_0}. Namely,
\begin{equation}\label{eqQweight}
w(Q)=\prod_{k=1}^gt_k^{-1}=\prod_{k=1}^gs_k^{-k}.
\end{equation}
Since the polynomial $Q$ is homogeneous of degree $g(g+1)/2$ in the
coefficients of $W_0$, we see that the $Q$-invariant of any $B\in
(s)\cdot XE$ is bounded by $X^{g(g+1)/2}w(Q)$, up to a multiplicative
constant depending only on $\GG_1$ and $R$.

Although we shall not use this fact directly, 
the points in $((s)\cdot XE)\cap \frac12W_{00}(\Z)$ with irreducible
invariant polynomial occur predominantly when the coordinates $s_i$
are so large that they force any (half-)integral point of $(s)\cdot
XE$ to lie inside $W_{00}$; this can be deduced by following the  
proof of Proposition~\ref{prop:bg2} and using Proposition~\ref{propodd}. 
On the other hand, since the weight of
the $Q$-invariant is a product of negative powers of $s_i$, the
$Q$-invariants of points in $((s)\cdot XE)\cap \frac12W_{00}(\Z)$
become large when the coordinates $s_i$ are small.  It is the tension
between these two requirements that underlies the proof of the following proposition, 
which gives the desired power-saving estimate for the
number of elements in $((s)\cdot XE) \cap \frac12W_{00}(\Z)$ having
large $|Q|$-invariant.

\begin{proposition}\label{proplargeQbound}
We have
\begin{equation*}
\int_{s_i\gg 1} \#\{B\in ((s)\cdot XE)\cap \LL\cap
{\textstyle\frac12} W_{00}(\Z):|Q|(B)>M\}\delta(s)d^\times s=O(\frac{1}{M}X^{n(n+1)/2}\log X).
\end{equation*}
\end{proposition}

\begin{proof}
First note that if an element in $\frac12 W_{00}(\Z)$ has
$(i,j)$-coordinate $0$ for some $i+j=n$, then the element has
discriminant $0$ and hence is not in $\LL$. Since the weight of $b_{i,n-i}$ is $s_i^{-1}$, to count points in $\LL$ it suffices to integrate only in the region where $s_i\ll X$ for all $i$.
Furthermore, the condition on the $|Q|$-invariant implies that
it suffices to integrate only in the region where $X^{g(g+1)/2}w(Q) \gg M$.

Let $S$ denote the set of coordinates of $W_{00}$, i.e.,
$S=\{b_{ij}:i+j\geq n\}$. For $(s)$ in the range $1\ll s_i\ll X$, we
have $Xw(\alpha)\gg 1$ for all $\alpha\in S$. Hence the number of lattice points in $(s)\cdot (XE)$ for $(s)$ in this range is $\ll \prod_{\alpha\in S}(Xw(\alpha))$. Therefore, 
\begin{eqnarray*}
&&\displaystyle\int_{s_i\gg 1} \#\{B\in((s)\cdot (XE))\cap \LL\cap {\textstyle\frac12} W_{00}(\Z):|Q(B)|>M\}\delta(s)d^\times s \\
&\ll&\displaystyle\int_{1\ll s_i\ll X,\,\,X^{g(g+1)/2}w(Q) \gg M} \prod_{\alpha\in S}\bigl(Xw(\alpha)\bigr)\delta(s)d^\times s \\
&\ll&\displaystyle\int_{1\ll s_i\ll X,\,\,X^{g(g+1)/2}w(Q) \gg M} X^{n(n+1)/2-g^2}\prod_{k=1}^gs_k^{2k-1}d^\times s \\
&\ll&\displaystyle\frac{1}{M}\int_{s_i=1}^X X^{n(n+1)/2-g^2+g(g+1)/2}w(Q)\prod_{k=1}^gs_k^{2k-1}d^\times s \\
&\ll&\displaystyle\frac{1}{M}\int_{s_i=1}^X X^{n(n+1)/2-g(g-1)/2}\prod_{k=1}^gs_k^{k-1}d^\times s \\
&\ll&\displaystyle\frac{1}{M}X^{n(n+1)/2}\log(X),
\end{eqnarray*}
where the second inequality follows from the definition
\eqref{eqhaarodd} of $\delta(s)$ and the computation \eqref{wbij} of
the weights of the coordinates $b_{ij}$, the third inequality follows
from the fact that $X^{g(g+1)/2}w(Q) \gg M$, the fourth inequality
follows from the computation of the weight of $Q$ in
\eqref{eqQweight}, and the $\log X$ factor comes from the integration
over $s_1$.
\end{proof}

The estimate in Theorem \ref{thm:mainestimate}(b) for odd $n$ now
follows from Theorem \ref{keymaporbit} and
Propositions~\ref{prop:bg2}, \ref{propodd} and~\ref{proplargeQbound},
in conjunction with the bound on the number of reducible polynomials
proved in Proposition~\ref{propredboundall}.

\section{A uniformity estimate for even degree monic polynomials}\label{sec:moniceven}

In this section, which is structured similarly to Section~\ref{sec:monicodd}, we
prove the estimate of Theorem~\ref{thm:mainestimate}(b) when $n=2g+2$
is even, for any $g\geq 1$.

\subsection{Invariant theory for the fundamental representation: $\SO_n$ on the space $W$ of symmetric $n\times n$ matrices}

Let $A_0$ denote the $n\times n$ symmetric matrix with $1$'s on the
anti-diagonal and $0$'s elsewhere. The group $\SO(A_0)$ acts on $W$
via the action
\begin{equation*}
\gamma\cdot B=\gamma B\gamma^t.
\end{equation*}
The central $\mu_2$ acts trivially and so the action descends to an
action of $G=\SO(A_0)/\mu_2$.

We recall some of the arithmetic invariant theory of the
representation $W$ of $n\times n$ symmetric matrices of the
(projective) split orthogonal group $G=\PSO_n.$ See \cite{SW} for more
details. The ring of polynomial invariants over $\C$ is
freely generated by the coefficients of the {\it invariant
  polynomial} $$f_B(x):=(-1)^{g+1}\det(A_0x-B)$$ (See \cite[\S2.1, 2.2]{SW}). We define the {\it
  discriminant} $\Delta$ and {\it height} $H$ on $W$ as the discriminant
and height of the invariant polynomial.

Let $k$ be a field of characteristic not $2$.
For any monic polynomial $f(x)\in k[x]$ of degree~$n$ such
that $\Delta(f)\neq0$, let $C_f$ denote the smooth hyperelliptic curve
$y^2=f(x)$ of genus $g$ and let $J_f$ denote its Jacobian. Then $C_f$
has two rational non-Weierstrass points at infinity that are conjugate
by the hyperelliptic involution.
The stabilizer of an element $B\in W(k)$ with invariant polynomial $f(x)$
is isomorphic to $J_f[2](k)$ by~\cite[Proposition~2.33]{W}, and hence has cardinality at most $\#J_f[2](\bar k)=2^{2g}$, where $\bar k$ denotes a separable closure of $k$.

We say that an element (or the $G(k)$-orbit of an element) $B\in W(k)$ with $\Delta(B)\neq 0$ is
{\it distinguished} if there exists a flag $Y'\subset Y$ defined over
$k$ where $Y$ is $(g+1)$-dimensional isotropic with respect to~$A_0$
and $Y'$ is $g$-dimensional isotropic with respect to $B$. If $B$ is
distinguished, then the set of these flags is in bijection with
$J_f[2](k)$ by \cite[Proposition~2.32]{W},  and so it too has cardinality at most~$2^{2g}$.

In fact, it is known (see~\cite[Proposition~22]{BGWhyper}) that the elements of $J_f[2](k)$ are in natural bijection with the even degree factorizations of $f$ defined over $k$.  (Note that the number of such factorizations of $f$ over $\bar k$ is indeed $2^{2g}$.) In particular, if $f$ is irreducible over $k$ and does not factor as $g(x)\bar g(x)$ over some quadratic extension of $k$, then the group $J_f[2](k)$ is trivial.

Let $W_0$ be the subspace of $W$ consisting of matrices whose top left
$g\times (g+1)$ block is zero. Then elements $B$ in $W_0(k)$ with
nonzero discriminant are all distinguished since the $(g+1)$-dimensional
subspace $Y_{g+1}$ spanned by the first $g+1$ basis vectors is
isotropic with respect to $A_0$ and the $g$-dimensional subspace
$Y_g\subset Y_{g+1}$ spanned by the first $g$ basis vectors is
isotropic with respect to $B$.
Let $G_0$ be the parabolic subgroup of
$G$ consisting of elements $\gamma$ such that $\gamma^t$ preserves the
flag $Y_g\subset Y_{g+1}$. Then $G_0$ acts on $W_0$.

An element $\gamma\in G_0$ has the block matrix form
\begin{equation}\label{eq:G_0even}
\gamma=\left(\begin{array}{ccc}\gamma_1 & 0 & 0\\ \delta_1 & \alpha & 0\\ \delta_2 & \delta_3 & \gamma_2
\end{array}\right)\in\left(\begin{array}{ccc}M_{g\times g} & 0 & 0\\ M_{1\times g} & M_{1\times 1} & M_{1\times (g+1)}\\ M_{(g+1)\times g} & M_{(g+1)\times 1} & M_{(g+1)\times (g+1)}
\end{array}\right),
\end{equation}
and so $\gamma\in G_0$ acts on the top right $g\times (g+1)$ block
of an element $B\in W_0$ by
$$\gamma.B^\sub = \gamma_1B^\sub\gamma_2^t,$$ where we use the
superscript ``top'' to denote the top right $g\times (g+1)$ block of
any given element of $W_0$. We may thus view $(A_0^\sub,B^\sub)$ as an
element of the representation $V_g=2\otimes g\otimes (g+1)$. As
before, we define the $Q$-{\it invariant} of $B\in W_0$ as the
$Q$-invariant of $(A_0^\sub,B^\sub)$:
\begin{equation}\label{eqQBeven}
Q(B):=Q(A_0^\sub,B^\sub).
\end{equation}
Then the $Q$-invariant is a relative invariant for the action of $G_0$ on $W_0$, i.e.,
for any $\gamma\in G_0$ in the form \eqref{eq:G_0even}, we have by \eqref{eq:relQ},
\begin{equation}\label{eq:weightG_0even}
Q(\gamma.B) = \det(\gamma_1)^{g+1}\det(\gamma_2)^gQ(B) =
\det(\gamma_1)\alpha^{-g}Q(B).
\end{equation}

In fact, we may extend the definition of the
$Q$-invariant to an even larger subset of $W(\Q)$ than $W_0(\Q)$.
We have the following proposition.

\begin{proposition}\label{prop:extendQeven}
Let $B\in W_0(\Q)$ be an element whose invariant polynomial $f(x)$ is
irreducible over $\Q$ and, when $n\geq 4$, does not factor as $g(x)\bar{g}(x)$ over some quadratic extension of $\Q$. Then for every $B'\in W_0(\Q)$ such that $B'$ is $G(\Z)$-equivalent to $B$, we have
$Q(B')=\pm Q(B)$.
\end{proposition}
\begin{proof}
The assumption on the factorization property of $f(x)$ implies that $J_f[2](\Q)$ is trivial.
The proof is now identical to that of Proposition~\ref{prop:extendQ}.
\end{proof}

We may thus define the $|Q|$-{\it invariant} for any element $B\in W(\Q)$ that is $G(\Z)$-equivalent to some $B'\in W_0(\Q)$ and whose invariant polynomial is irreducible over $\Q$ and does not factor as $g(x)\bar{g}(x)$ over any quadratic extension of $\Q$; indeed, we set $|Q|(B):=|Q(B')|$. By Proposition~\ref{prop:extendQeven}, this definition of $|Q|(B)$ is independent of  the choice of $B'$. We note again that all such elements $B\in W(\Q)$ are distinguished.

\subsection{Embedding $\U_m^\w$ into $\frac14W(\Z)$}\label{sembedeven}

Let $m$ be a positive squarefree integer and let $f$ be an monic
integer polynomial whose discriminant is weakly divisible by
$m^2$. Then as proved in \S\ref{sembedodd}, there exists an integer
$\ell$ such that $f(x+\ell)$ has the form
$$f(x+\ell) = x^n + c_1x^{n-1} + \cdots + c_{n-2}x^2 + mc_{n-1}x + m^2c_n.$$

Consider the following matrix:
\begin{equation}\label{mat2}
 B_m(c_1,\ldots,c_n) =
 \left(\!\!\!\!\!\begin{array}{cccccccccc}&&&&&&&&m&0\\[.065in]&&&&&&&\iddots&\:\;\iddots& \\[.025in]&&&&&&1&\;\;\iddots&&\\[.185in] &&&&&\,1&0&&&\\[.185in] &&&&\:\;1\;&-c_1/2&&&& \\[.175in]&&&1&\!\!-c_1/2\,&\!\!c_1^2/4\!-\!c_2\!\!&-c_3/2&&&\\[.175in]&&1&\;\;\;0\;\;\;&&-c_3/2&-c_4&\!\!\!-c_5/2\!\!\!&&\\[.085in]&\iddots&\;\;\:\iddots\:\;\;&&&&-c_5/2&-c_6&\ddots&\\[.0125in] \;\;\;m\;\;\:&\;\,\,\iddots&&&&&&\ddots&\ddots&\!\!-c_{n-1}/2\!\!
   \\[.125in] 0&&&&&&&&\!\!\!\!-c_{n-1}/2\!\!\!\!&-c_n \end{array}\!\right).
\end{equation}
It follows from a direct computation that
$$f_{B_m(c_1,\ldots,c_n)}(x) = x^n + c_1x^{n-1} + \cdots + c_{n-2}x^2 +
mc_{n-1}x + m^2c_n.$$ We set $\sigma_m(f) := B_m(c_1,\ldots,c_n) +
\ell A_0\in \frac14 W(\Z)$. Then evidently $f_{\sigma_m(f)}=f.$ A direct computation again shows that
$Q(B_m(c_1,\ldots,c_n))=m$. Since the $Q$-invariant on $2\otimes
g\otimes (g+1)$ is $\SL_2$-invariant, we conclude
that $$Q(\sigma_m(f))=m.$$

\begin{theorem}\label{th:mapeven}
Let $m$ be a positive squarefree integer. There exists a map
$\sigma_m:\U_m^\w\to \frac14W(\Z)$ such that for every $f\in \U_m^\w$,
\begin{equation}\label{eq:liftQe}
f_{\sigma_m(f)}=f,\qquad Q(\sigma_m(f)) = m.
\end{equation}
\end{theorem}

Let $\LL$ be the set of elements $v\in \frac14W(\Z)$ that are $G(\Z)$-equivalent to some elements of $\frac14 W_0(\Z)$ and such that
the invariant polynomial of $v$ is irreducible
over $\Q$ and does not factor as $g(x)\bar{g}(x)$ over some quadratic extension of $\Q$.
Then by the remark following Proposition~\ref{prop:extendQeven}, we may view $|Q|$ as a function also on $\LL$.
Let $\U_m^{\w,\irr}$ denote the set of polynomials in $\U_m^\w$
that are irreducible
over $\Q$ and do not factor as $g(x)\bar{g}(x)$ over any quadratic extension of $\Q$. Then we have the following immediate consequence of Theorem~\ref{th:mapeven}:

\begin{theorem}\label{keymaporbiteven}
Let $m$ be a positive squarefree integer. There exists an injective map
$$\bar{\sigma}_m:\U_m^{\w,\irr}\to G(\Z)\backslash\LL$$ such that
$f_{\bar{\sigma}_m(f)}=f$ for every $f\in \U_m^\w$. Furthermore, every element
in every orbit in the image of $\bar{\sigma}_m$ has $|Q|$-invariant $m$.
\end{theorem}

The number of monic integer polynomials having height less than $X$ that are reducible or factor as $g(x)\bar{g}(x)$ over some quadratic extension of $\Q$ is of a strictly smaller order of magnitude than the total number of such polynomials (see, e.g., Proposition~\ref{propredboundall}). Thus to prove Theorem~\ref{thm:mainestimate}(b), it suffices to count the number of elements in
$\U_m^{\w,\irr}$ having height less than $X$ over all $m>M$, which, by Theorem~\ref{keymaporbiteven}, we may do by counting
$G(\Z)$-orbits on $\LL\subset \frac14W(\Z)$ having height less than $X$ and $|Q|$-invariant greater than $M$.
More precisely, let $N(\LL;M;X)$ denote the number of $G(\Z)$-equivalence classes of elements in
$\LL$ whose $|Q|$-invariant is greater than $M$ and whose height is
less than $X$. We obtain a bound for $N(\LL;M;X)$ using the same method as in Section 2.

\subsection{Counting $G(\Z)$-orbits in $\frac14W(\Z)$}\label{scoeffeven}

The counting problem for the representation $W$ of $G$ is studied in \cite{SW}. In this section, we recall some of the set up and results of \cite{SW}.

Let $R$ be a fundamental domain for the action of $G(\R)$ on the elements of $W(\R)$ having nonzero discriminant and height bounded by $1$ as constructed in \cite[\S4.1]{SW}. Let $\FF$ be a fundamental set for the left multiplication action of $G(\Z)$ on $G(\R)$ obtained using the Iwasawa decomposition of $G(\R)$. More explicitly, we have
$$
G(\R)=N(\R)TK,
$$ where $N$ is a unipotent group consisting of lower triangular matrices, $K$ is compact, and $T$ is the split torus of $G$ given by
\begin{equation*}
T=
\left\{\left(\begin{array}{ccccccc}
 t_1^{-1}&&&&&&\\
&\ddots &&&&&  \\
 && t_{g+1}^{-1} &&&&\\
 &&&& \!\!\!\!t_{g+1} &&\\
&&&&&\!\!\!\!\ddots &  \\
& &&&&& \!\!t_{1}
\end{array}\right)\right\}.
\end{equation*}
We may also make the following change of variables.
For $1\leq
i\leq g$, define $s_i$ to be
$$
s_i=t_i/t_{i+1},
$$
and let $s_g=t_gt_{g+1}$.
We denote an element of $T$ with coordinates $t_i$
(resp.\ $s_i$) by $(t)$ (resp.\ $(s)$). We may take $\FF$ to be contained in a {\it Siegel set}, i.e., contained in $N'T'K$,
where $N'$ consists of elements in $N(\R)$ whose coefficients are
absolutely bounded and $T'\subset T$ consists of elements in $(s)\in
T$ with $s_i\geq c$ for some positive constant $c$.

Let $\GG_1$ be a compact left $K$-invariant set in $G(\R)$ which is the closure of nonempty open set. Then as in Section 2, we have
\begin{equation}
  N(\LL;M;X)\ll \int_{\gamma\in\FF} \#\{B\in((\gamma \GG_1)\cdot (XR))\cap\LL:|Q|(B)>M\}
  d\gamma,
\end{equation}
where the implied constant depends only on $\GG_1$ and $R$, and where $d\gamma$ is a Haar measure on $G(\R)$ given by
$$
d\gamma=dn\,\delta(s)d^\times s\,dk,
$$ where $dn$ is a Haar measure on the unipotent group $N(\R)$,
$dk$ is a Haar measure on the compact group~$K$, $d^\times s$ is
given by
$$
d^\times s:=\prod_{i=1}^{g+1}\frac{ds_i}{s_i},
$$
and
$\delta(s)$ is given by
\begin{equation}\label{eqhaareven}
  \delta(s)=\prod_{k=1}^{g-1} s_k^{k^2-2kg-k}\cdot (s_gs_{g+1})^{-g(g+1)/2};
\end{equation}
see~\cite[(20)]{SW}.

Since $s_i\geq c$ for every $i$, there exists a compact subset $N''$ of
$N(\R)$ containing $(t)^{-1}N'\,(t)$ for all $t\in T'$. Since $N''$, $K$, $G_0$ are compact and $R$ is bounded, the set $E=N''KG_0R$ is bounded. Then we have
\begin{equation}\label{eq:REe}
N(\LL;M;X)\ll \int_{s_i\gg 1} \#\{B\in((s)\cdot XE)\cap\LL:|Q|(B)>M\}\delta(s)d^\times s.
\end{equation}

As before, we denote the coordinates of $W$ by $b_{ij}$, for $1\leq
i\leq j\leq n$, and we denote
the $T$-weight of a coordinate $\alpha$ on $W$, or a
product $\alpha$ of powers of such coordinates, by $w(\alpha)$.
We compute the weights of the coefficients $b_{ij}$ to be
\begin{equation}\label{wbij2}
w(b_{ij})=\left\{
\begin{array}{rcl}
t_i^{-1}t_j^{-1} &\mbox{ if }& i,j\leq g+1\\
t_i^{-1}t_{n-j+1} &\mbox{ if }& i\leq g+1,\; j>g+1\\
t_{n-i+1}t_{n-j+1} &\mbox{ if }& i,j>g+1.
\end{array}
\right.
\end{equation}
Then the $(i,j)$-entry of any $B\in (s)\cdot XE$ is bounded by $Xw(b_{ij})$, up to a multiplicative constant depending only on $\GG_1$ and $R$.

Let $W_{00}\subset W$ denote the space of symmetric matrices $B$
such that $b_{ij}=0$ for $i+j < n$. Let $W^\dist$ denote the subset of $W(\Q)$ consisting of $\Q$-distinguished elements. Then $\LL$ is a subset of $W^\dist$. It was shown in \cite[\S4.2]{SW}
that most of the lattice points in $W^\dist$ lie inside $W_{00}$. More precisely, we have the following estimates from \cite[Proposition 21 and 23]{SW}.

\begin{proposition}\label{prop:sw6}
We have
\begin{eqnarray}
  \int_{s_i\gg 1} \#\{B\in((s)\cdot XE)\cap({\textstyle\frac14} W(\Z)\setminus{\textstyle\frac14} W_{00}(\Z)):b_{11}=0\}\delta(s)d^\times s &=& O_\epsilon(X^{n(n+1)/2-1+\epsilon}) \\
  \label{eq:selbergpree}\int_{s_i\gg 1} \#\{B\in((s)\cdot XE)\cap\LL:b_{11}\neq 0\}\delta(s)d^\times s &=& o(X^{n(n+1)/2}).
\end{eqnarray}
\end{proposition}

Therefore, to prove Theorem~\ref{thm:mainestimate}(b) when $n=2g+2$ is even, it remains to obtain a power saving improvement of \eqref{eq:selbergpree} and estimate
\begin{equation}\label{eq:Qinve}
\int_{s_i\gg 1} \#\{B\in ((s)\cdot XE)\cap \LL\cap
{\textstyle\frac14} W_{00}(\Z):|Q|(B)>M\}\delta(s)d^\times s.
\end{equation}

\subsection{Proof of Theorem~\ref{thm:mainestimate}(b) for even $n$}\label{sgomeven}

We begin with a power saving improvement of \eqref{eq:selbergpree}.

\begin{proposition}\label{propeven}
  We have
  \begin{equation*}
\displaystyle\int_{s_i\gg 1} \#\{B\in ((s)\cdot XE)\cap W^\dist \cap {\textstyle\frac14} W(\Z): b_{11}\neq 0\}\delta(s)d^\times s =O_\epsilon(X^{n(n+1)/2-1/5+\epsilon}).
  \end{equation*}
\end{proposition}
\begin{proof}
In the proof of \cite[Proposition 23]{SW}, it is shown that the
$p$-adic density of elements in $W(\Z_p)$ that are
$\Q_p$-distinguished is bounded uniformly away from $1$.  Then an
application of the Selberg sieve exactly as in \cite{ShTs} yields the
result.
\end{proof}

We now estimate \eqref{eq:Qinve}. The $Q$-invariant can also be given
a weight by viewing the torus $T$ as sitting inside $G_0$ and using
\eqref{eq:weightG_0even}. Namely,
\begin{equation}\label{eqQweighteven}
w(Q)=(t_1\cdots t_g)^{-1}t_{g+1}^g=\prod_{k=1}^g s_k^k.
\end{equation}
Since the polynomial $Q$ is homogeneous of degree $g(g+1)/2$ in the
coefficients of $W_0$, we see that the $Q$-invariant of any $B\in
(s)\cdot XE$ is bounded by $X^{g(g+1)/2}w(Q)$, up to a multiplicative
constant depending only on $G_0$ and $R$.

\begin{proposition}\label{proplargeQboundeven}
We have
\begin{equation*}
\int_{s_i\gg 1} \#\{B\in ((s)\cdot(XE))\cap \LL\cap {\textstyle\frac14} W_{00}(\Z):|Q|(B)>M\}\delta(s)d^\times s=O(X^{n(n+1)/2}\log^2 X/M).
\end{equation*}
\end{proposition}
\begin{proof}
Analogous to the proof of
Proposition \ref{proplargeQbound}, in order for the set
$\{B\in((s)\cdot (XE))\cap \LL\cap \frac14 W_{00}(\Z):|Q|(B)>M\}$ to be nonempty,
the following conditions must be satisfied:
\begin{equation}\label{eqscond}
\begin{array}{rcl}
Xs_i^{-1}&\gg& 1,\\[.1in]
Xs_gs_{g+1}^{-1}&\gg& 1,\\[.1in]
X^{g(g+1)/2}w(Q)&\gg& M.
\end{array}
\end{equation}

Let $S$ denote the set of coordinates of $W_{00}$, i.e.,
$S=\{b_{ij}:i+j\geq n\}$. Let $T_{X,M}$ denote the set of~$(s)$
satisfying $s_i\gg 1$ and the conditions of \eqref{eqscond}. Then we have
\begin{eqnarray*}
&&\displaystyle\int_{s_i\gg 1} \#\{B\in((s)\cdot (XE))\cap \LL\cap {\textstyle\frac12} W_{00}(\Z):|Q|(B)>M\}\delta(s)d^\times s \\
&\ll&\displaystyle\int_{(s)\in T_{X,M}} \big(\prod_{\alpha\in S}(Xw(\alpha))\big)\delta(s)d^\times s \\
&\ll&\displaystyle\int_{(s)\in T_{X,M}} X^{n(n+1)/2-g(g+1)}\prod_{k=1}^{g-1}s_k^{2k-1}\cdot s_g^{g-1}s_{g+1}^{g}d^\times s \\
&\ll&\displaystyle\frac{1}{M}\int_{(s)\in T_{X,M}} X^{n(n+1)/2-g(g+1)/2}w(Q)\prod_{k=1}^{g-1}s_k^{2k-1}\cdot s_g^{g-1}s_{g+1}^{g}d^\times s \\
&\ll&\displaystyle\frac{1}{M}\int_{(s)\in T_{X,M}} X^{n(n+1)/2-g(g+1)/2}\prod_{k=1}^{g-1}s_k^{k-1}\cdot s_g^{-1}s_{g+1}^{g}d^\times s \\
&\ll&\displaystyle\frac{1}{M}\int_{(s)\in T_{X,M}} X^{n(n+1)/2-g(g+1)/2+g}\prod_{k=1}^{g-1}s_k^{k-1}\cdot s_g^{g-1}d^\times s \\
&\ll&\displaystyle\frac{1}{M}X^{n(n+1)/2}\log^2(X),
\end{eqnarray*}
where the first inequality follows from the fact that $Xw(b_{ij})\gg1$ for all $b_{ij}\in S$ when $(s)$ is in the range $1\ll s_i\ll X$, the second inequality follows from the definition \eqref{eqhaareven} of $\delta(s)$ and the computation (\ref{wbij2}) of the weights of the coordinates $b_{ij}$, the third inequality follows from the fact that $X^{g(g+1)/2}w(Q) \gg M$,
the fourth inequality follows from the computation of the weight
of $Q$ in \eqref{eqQweighteven}, the fifth inequality comes from multiplying by the factor $(Xs_gs_{g+1}^{-1})^g\gg1$, and  the $\log^2
X$ factor in the last inequality comes from the integrals over $s_1$ and $s_{g+1}$.
\end{proof}

The estimate in Theorem \ref{thm:mainestimate}(b) for even $n$ now
follows from Theorem \ref{keymaporbiteven} and Propositions~\ref{prop:sw6},
\ref{propeven} and \ref{proplargeQboundeven}, in conjunction with the
bound on the number of reducible polynomials proved in
Proposition~\ref{propredboundall}.

\section{Proof of the main theorems}\label{sec:sieve}

In this section, we prove a result from which Theorems \ref{polydisc2}
and \ref{polydiscmax2} immediately follow. Let $\Sigma=(\Sigma_v)_v$
be a collection of sets $\Sigma_v\subset \Vmn(\Z_v)$ indexed by places
$v$ of $\Q$, such that $\Sigma_p$ is defined by congruence conditions
modulo some power of $p$ for any finite prime $p$ and $\Sigma_\infty\subset\Vmn(\R)$ consists
of all degree-$n$ polynomials $f$ such that the number of real roots
of $f$ lies in some fixed nonempty subset of
$\{0,\ldots,n\}$. Such a set is called a collection of local specifications. Associate to each collection $\Sigma$ the subset
$\V(\Sigma)$ of $\Vmn(\Z)$ consisting of all elements $f$ that satisfy
$f\in\Sigma_v$ for all places $v$.  For any positive integer $\kappa$, we say that a collection
$\Sigma$ of local specifications is $\kappa$-{\it acceptable} if $\Sigma_p$ is defined by congruence conditions modulo $p^\kappa$ for all primes $p$; and for all
sufficiently large primes $p$, the sets $\Sigma_p$ contain every
element $f\in\Vmn(\Z_p)$ with $p^2\nmid\Delta(f)$. The local specifications corresponding to Theorems \ref{polydisc2}
and \ref{polydiscmax2} are $2$-acceptable. For a set $S\subset \Vmn(\Z)$, let $S_X$
denote the set of elements in $S$ with height bounded by~$X$. Then we have the
following theorem:

\begin{theorem}\label{thmaingencong}
Let $\kappa$ be a positive integer and let $\Sigma$ be a $\kappa$-acceptable collection of local specifications. Then
\begin{equation*}
\#\V(\Sigma)_X=\Vol(\Sigma_{\infty,H<1})\prod_p\Vol(\Sigma_p)\cdot
X^{n(n+1)/2}+O_\epsilon(X^{n(n+1)/2-\min\{1/5,1/(2\kappa)\}+\epsilon}),
\end{equation*}
where $\Sigma_{\infty,H<1}$ is the set of elements in $\Sigma_\infty$
having height less than $1$, the volumes of sets in $\Vmn(\Z_p)$
$($resp.\ $\Vmn(\R))$ are computed with respect to the Haar-measures
normalized so that $\Vmn(\Z_p)$ has volume~$1$ $($resp.\ $\Vmn(\Z)$ has
covolume~1$)$, and where the implied constant depends only on $n$ and $\Sigma$.
\end{theorem}

This section is organized as follows. First in \S\ref{sec:sieve}.1 we
prove some estimates for the number of reducible elements in
$\Vmn(\Z)$ having bounded height.  Then in \S\ref{sec:sieve}.2, we
prove a uniformity estimate on polynomials whose discriminants are
divisible by a large square.  We then prove Theorem~\ref{thmaingencong} by
using this uniformity estimate and a squarefree sieve, from which we then deduce Theorems \ref{polydisc2} and
\ref{polydiscmax2}.

\subsection{Estimates on reducible forms}

Let $\Vmn(\Z)$ denote the set of monic integer polynomials of
degree~$n$. Let $\Vmn(\Z)^\red$ denote the subset of polynomials that
are reducible or when $n\geq 4$, factor as $g(x)\bar{g}(x)$ over some
quadratic extension of~$\Q$.   We first
give a power saving bound for the number of polynomials in
$\Vmn(\Z)^\red$ having bounded height. We start with the following
lemma.

\begin{lemma}\label{lemlin}
The number of elements in $\Vmn(\Z)_X$ that have a rational linear factor
is bounded by $O(X^{n(n+1)/2-n+1}\log X)$.
\end{lemma}
\begin{proof}
  Consider the polynomial
\begin{equation*}
  f(x)=x^n+a_{1}x^{n-1}+\cdots +a_n\in \Vmn(\Z)_X.
\end{equation*}
First, note that the number of such polynomials with $a_n=0$ is bounded
by $O(X^{n(n+1)/2-n})$. Next, we assume that $a_n\neq 0$. There are
$O(X^{n(n+1)/2-n+1})$ possibilities for the $(n-1)$-tuple
$(a_1,a_2,\ldots,a_{n-2},a_n)$. If $a_n\neq 0$ is fixed, then there
are $O(\log X)$ possibilities for the linear factor $x-r$ of $f(x)$,
since $r\mid a_n$. By setting $f(r)=0$, we see that the
values of $a_1,a_2,\ldots,a_{n-2},a_n$, and $r$ determine $a_{n-1}$
uniquely. The lemma follows.
\end{proof}

Following arguments of Dietmann~\cite{Dietmann}, we now prove that the number of reducible monic
integer polynomials of bounded height is negligible, with a power-saving error term.

\begin{proposition}\label{propredboundall}
We have
  \begin{equation*}
\#\Vmn(\Z)^\red_X=O(X^{n(n+1)/2-n+1}\log X).
  \end{equation*}
\end{proposition}
\begin{proof}
  First, by \cite[Lemma~2]{Dietmann}, we have that
\begin{equation}\label{eqtemppoly}
  x^n+a_1x^{n-1}+\cdots +a_{n-1}x+t
\end{equation}
has Galois group $S_n$ over $\Q(t)$ for all $(n-1)$-tuples
$(a_1,\ldots,a_{n-1})$ aside from a set $S$ of cardinality
$O(X^{(n-1)(n-2)/2})$. Hence, the number of $n$-tuples
$(a_1,\ldots,a_n)$ with height bounded by $X$ such that the Galois
group of $x^n+a_1x^{n-1}+\cdots +a_{n-1}x+t$ over $\Q(t)$ is not $S_n$
is $O(X^{(n-1)(n-2)/2}X^n) = O(X^{n(n+1)/2-n+1})$.

Next, let $H$ be a subgroup of $S_n$ that arises as the Galois group
of the splitting field of a polynomial in $\Vmn(\Z)$ with no rational
root. For reducible polynomials, we have from \cite[Lemma~4]{Dietmann} that
$H$ has index at least $n(n-1)/2$ in $S_n$. When $n\geq4$ is even and
the polynomial factors as $g(x)\bar{g}(x)$ over a quadratic extension,
the splitting field has degree at most $2(n/2)!$ and so the index of
the corresponding Galois group in $S_n$ is again at least $n(n-1)/2$.
For fixed $a_1,\ldots,a_{n-1}$ such that the polynomial
\eqref{eqtemppoly} has Galois group $S_n$ over $\Q(t)$, an argument
identical to the proof of \cite[Theorem~1]{Dietmann} implies that the
number of $a_n$ with $|a_n|\leq X^n$ such that the Galois group of the
splitting field of $x^n+a_1x^{n-1}+\cdots a_n$ over $\Q$ is $H$ is
bounded by
$$
O_\epsilon\Bigl(X^\epsilon\exp\Bigl(\frac{n}{[S_n:H]}\log X +O(1)\Bigr)\Bigr)
=O(X^{2/(n-1)+\epsilon}).
$$
In conjunction with Lemma~\ref{lemlin}, we thus obtain the estimate
  \begin{equation*}
    \#\Vmn(\Z)^\red_X=O(X^{n(n+1)/2-n+1}\log X)+O(X^{n(n+1)/2-n+1})+
    O_\epsilon(X^{n(n+1)/2-n+2/(n-1)+\epsilon}),
  \end{equation*}
and the proposition follows.
\end{proof}


\subsection{Proof of Theorem \ref{thmaingencong}}

Recall that we proved the estimates of Theorem
\ref{thm:mainestimate}(b) for odd and even $n$ in \S2 and \S3,
respectively. The estimate of Theorem \ref{thm:mainestimate}(a) is a
direct consequence of \cite[Theorem 3.5, Lemma 3.6]{geosieve} since
the discriminant polynomial on $\Vmn$ is irreducible.  For any
positive squarefree integer $m$, let $\U_m$ denote the set of all
elements in $\Vmn(\Z)$ whose discriminants are divisible by $m^2$. We
now prove the following direct consequence of Theorem \ref{thm:mainestimate}.

\begin{theorem}\label{unif}
Let $\U_{m,X}$ denote the set of elements in $\U_m$ having height
bounded by $X$. For any positive real number $M$, we have
\begin{equation}\label{equ2}
\sum_{\substack{m>M\\ m\;\mathrm{ squarefree}}}
\#\U_{m,X}=O_\epsilon(X^{n(n+1)/2+\epsilon}/\sqrt{M})+O_\epsilon(X^{n(n+1)/2-1/5+\epsilon}).
\end{equation}
\end{theorem}
\begin{proof}
Note first that an element $f\in \Vmn(\Z)$ belongs to at most
$O(X^\epsilon)$ different sets $\U_m$, since $m$ is a divisor of
$\Delta(f)$. Hence it suffices to prove the bound \eqref{equ2} for the
number of elements in the union of $\U_{m,X}$ over all squarefree
integers $m>M$. Now suppose $f$ is an element in this union. Let $k_1$
denote the product of all the primes $p$ where $p^2$ strongly divides
$\Delta(f)$ and let $k_2$ denote the product of all the primes $p$
where $p^2$ weakly divides $\Delta(f)$. Then $f\in
\U_{k_1}^{(1)}\cap\U_{k_2}^{(2)}$ with $k_1k_2>M$. In other words,
$$\bigcup_{\substack{m>M\\ m\;\mathrm{ squarefree}}} \U_{m,X} \subset \bigcup_{\substack{k_1>\sqrt{M}\\ k_1\;\mathrm{ squarefree}}} \U_{k_1,X}^{(1)} \cup \bigcup_{\substack{k_2>\sqrt{M}\\ k_2\;\mathrm{ squarefree}}} \U_{k,X}^{(2)}.$$
The theorem now follows from Parts (a) and (b) of Theorem
\ref{thm:mainestimate}.
\end{proof}

We remark that the $\sqrt{M}$ in the denominator in \eqref{equ2} can
be improved to $M$. However we will be using Theorem \ref{unif} for
$M=X^{1/\kappa}$ (where $\kappa=2$ for the application to Theorems
\ref{polydisc2} and \ref{polydiscmax2}) in which case the second term
$O_\epsilon(X^{n(n+1)/2-1/5+\epsilon})$ sometimes dominates. We
outline here how to improve the denominator to $M$ for the sake of
completeness. Break up $\U_m$ into sets $\U_{m_1}^\s\cap\U_{m_2}^\w$
for positive squarefree integers $m_1,m_2$ with $m_1m_2=m$ as
above. Break the ranges of $m_1$ and $m_2$ into dyadic ranges. For
each range, we count the number of elements in
$\U_{m_1}^\s\cap\U_{m_2}^\w$ by embedding each $\U_{m_2}^\w$ into
$\frac14W(\Z)$ as in Sections \ref{sec:monicodd} and
\ref{sec:moniceven}. Earlier, we bounded the cardinality of the image
of $\U_{m_2,X}^\w$ by splitting $\frac14W(\Z)$ up into two pieces:
$\frac14W_{00}(\Z)$ and $\frac14W(\Z)\setminus\frac14W_{00}(\Z)$. The
bound on the second piece does not depend on $m_2$ and continues to be
$O_\epsilon(X^{n(n+1)/2-1/5+\epsilon})$. However for the first piece,
we now impose the further condition that elements in $\frac14
W_{00}(\Z)$ are strongly divisible by $p^2$ for all prime factors $p$
of $m_1$ and apply the quantitative version of the Ekedahl sieve as
in~\cite{geosieve}. This gives the desired additional $1/m_1$ saving,
improving the bound to
\begin{equation*}
O_\epsilon(X^{n(n+1)/2+\epsilon}/M)+O_\epsilon(X^{n(n+1)/2-1/5+\epsilon}).
\end{equation*}
The reason for counting in dyadic ranges of $m_1$ and $m_2$ is
that for both the strongly and weakly divisible cases, we count
not for a fixed $m$ but sum over all $m>M$.

\vspace{10pt}

Let $(\Sigma_v)_v$ be a $\kappa$-acceptable collection of local
specifications. Let $N$ denote a positive integer such that for every
prime $p>N$, the set $\Sigma_p$ contains every element $f\in
\Vmn(\Z_p)$ with $p^2\nmid \Delta(f)$. Let $P$ denote the product of
all primes $p\leq N$. For squarefree integers $m$, we let
$\W_m(\Sigma)$ denote the set of elements $f$ such that
$f\not\in\Sigma_p$ for all $p\mid m$ and let $m'$ denote the product
of all the prime factors of $m$ that are larger than $N$. Then $m'\geq
m/P$ and $\W_m(\Sigma)\subset \W_{m'}.$ Since $P$ depends only on
$\Sigma$, we may assume that $\log X > P$ in what follows. In other
words, we have
\begin{equation}\label{eq:haha}
\sum_{\substack{m>X^{1/\kappa}\\ m\;\mathrm{ squarefree}}}\#\W_{m}(\Sigma)_X \leq \sum_{\substack{m'>X^{1/\kappa-\epsilon}\\ m'\;\mathrm{ squarefree}}}\#\W_{m',X}.
\end{equation}

For each prime $p$, let $\theta(p)$ denote
$\Vol(\Sigma_p)$, let $\theta(\infty)$ denote
$\Vol(\Sigma_{\infty,H<1})$, and set
$\bar{\theta}(p):=1-\theta(p)$. We define
$\bar{\theta}(m)=\prod_{p\mid m}\bar{\theta}(m)$ for squarefree
integers $m$.   Let $\mu$ denote the
M\"obius function. We have
\begin{equation}\label{eqth1}
\begin{array}{rcl}
\#\V(\Sigma)_X&=&\displaystyle\sum_{m\geq 1}\mu(m)\#\W_m(\Sigma)_X\\[.2in]
&=&\displaystyle\sum_{m=1}^{X^{1/\kappa}}\mu(m)\theta(\infty)\bar{\theta}(m)X^{n(n+1)/2}
+O\Bigl(\sum_{m=1}^{X^{1/\kappa}}X^{n(n+1)/2-n}\Bigr)
+O\Bigl(\sum_{m>X^{1/\kappa}}\#\W_{m}(\Sigma)_X\Bigr)\\[.2in]
&=&\displaystyle\theta(\infty)\prod_p\theta(p)\cdot X^{n(n+1)/2}+O_\epsilon(X^{n(n+1)/2-1/(2\kappa)+\epsilon})+O_\epsilon(X^{n(n+1)/2-1/5+\epsilon}),
\end{array}
\end{equation}
where the final equality follows from \eqref{eq:haha} and Theorem \ref{unif}.
This concludes the proof of Theorem~\ref{thmaingencong}.



Finally note that Theorems \ref{polydisc2} and \ref{polydiscmax2}
follow from Theorem \ref{thmaingencong} since the corresponding
families are $2$-acceptable, and the constants $\lambda_n$ and
$\zeta(2)^{-1}$ appearing in these theorems are equal simply to
$\prod_p\lambda_n(p)$ and $\prod_p\rho_n(p)$, respectively.

\section{A lower bound on the number of degree-$n$ number fields that are monogenic / have a short generating vector}\label{latticearg}

Let $g\in \Vmn(\R)$ be a monic real polynomial of degree $n$ and
nonzero discriminant with $r$ real roots and $2s$ complex roots.  Then
$\R[x]/(g(x))$ is naturally isomorphic to $\R^n\cong \R^r\times \C^s$
as $\R$-vector spaces via its real and complex embeddings (where we
view $\C$ as $\R+\R\sqrt{-1}$).  The $\R$-vector space $\R[x]/(g(x))$
also comes equipped with a natural basis, namely
$1,\theta,\theta^2,\ldots,\theta^{n-1}$, where $\theta$ denotes the
image of $x$ in $\R[x]/(g(x))$. Let $R_g$ denote the lattice spanned
by $1,\theta,\ldots,\theta^{n-1}.$ In the case that $g$ is an integral
polynomial in $\Vmn(\Z)$, the lattice $R_g$ may be identified with the
ring $\Z[x]/(g(x))\subset\R[x]/(g(x))\subset \R^n$.

Since $g(x)$ gives a lattice in $\R^n$ in this way, we may ask whether
this basis is reduced in the sense of Minkowski, with respect to the
usual inner product on $\R^n$.\footnote{Recall that a $\Z$-basis
  $\alpha_1,\ldots,\alpha_n$ of a lattice $L$ is called {\it
    Minkowski-reduced} if successively for $i=1,\ldots, n$ the vector
  $\alpha_i$ is the shortest vector in $L$ such that
  $\alpha_1,\ldots,\alpha_i$ can be extended to a $\Z$-basis of
  $L$. Most lattices have a unique Minkowski-reduced basis.}  More
generally, for any monic real polynomial $g(x)$ of degree $n$ and
nonzero discriminant, we may ask whether the basis
$1,\theta,\theta^2,\ldots,\theta^{n-1}$ is Minkowski-reduced for the
lattice $R_g$, up to a unipotent upper-triangular transformation
over~$\Z$ (i.e., when the basis $[1\;\;\theta\;\;\theta^2\;\cdots\;
  \theta^{n-1}]$ is replaced by $[1\;\;\theta\;\;\theta^2\;\cdots\;
  \theta^{n-1}]A$ for some upper triangular $n\times n$ integer matrix
$A$ with $1$'s on the diagonal).

More precisely, given $g\in \Vmn(\R)$ of nonzero discriminant, let us
say that the corresponding basis
$1,\theta,\theta^2,\ldots,\theta^{n-1}$ of $\R^n$ is {\it
  quasi-reduced} if there exist monic integer polynomials $h_i$ of
degree~$i$, for $i=1,\ldots,n-1$, such that the basis
$1,h_1(\theta),h_2(\theta),\ldots,h_{n-1}(\theta)$ of $R_g$ is
Minkowski-reduced (so that the basis
$1,\theta,\theta^2,\ldots,\theta^{n-1}$ is Minkowski-reduced up to a unipotent
upper-triangular transformation over~$\Z$). By abuse of language, we
then call the polynomial $g$ {\it quasi-reduced} as well. We say that
$g$ is {\it strongly quasi-reduced} if in addition $\Z[x]/(g(x))$ has
a unique Minkowski-reduced basis.

The relevance of being strongly quasi-reduced is contained in the
following lemma.

\begin{lemma}\label{reducedg}
Let $g(x)$ and $g^*(x)$ be distinct monic integer polynomials of
degree $n$ and nonzero discriminant that are strongly quasi-reduced
and whose $x^{n-1}$-coefficients vanish.  Then $\Z[x]/(g(x))$ and
$\Z[x]/(g^*(x))$ are non-isomorphic rings.
\end{lemma}

\begin{proof}
Let $\theta$ and $\theta^*$ denote the images of $x$ in $\Z[x]/(g(x))$
and $\Z[x]/(g^*(x))$, respectively.  By the assumption that $g$ and
$g^\ast$ are strongly quasi-reduced, we have that
$1,h_1(\theta),h_2(\theta),\ldots,h_{n-1}(\theta)$ and
$1,h_1^*(\theta^*),h_2^*(\theta^*),\ldots,h^*_{n-1}(\theta^*)$ are
the unique Minkowski-reduced bases of $\Z[x]/(g(x))$ and
$\Z[x]/(g^*(x))$, respectively, for some monic integer polynomials
$h_i$ and $h_i^*$ of degree $i$ for $i=1,\ldots,n-1$.

If $\phi:\Z[x]/(g(x))\to\Z[x]/(g^*(x))$ is a ring isomorphism, then by
the uniqueness of Minkowski-reduced bases for these rings, $\phi$ must
map Minkowski basis elements to Minkowski basis elements, i.e.,
$\phi(h_i(\theta))=h_i^*(\theta^*)$ for all $i$.  In particular, this
is true for $i=1$, so $\phi(\theta)=\theta^*+c$ for some $c\in\Z$,
since $h_1$ and $h_1^*$ are monic integer linear polynomials.
Therefore $\theta$ and $\theta^*+c$ must have the same minimal
polynomial, i.e., $g(x)=g^*(x-c)$; the assumption that $\theta$ and
$\theta^*$ both have trace 0 then implies that $c=0$.  It follows that
$g(x)=g^*(x)$, a contradiction.  We conclude that $\Z[x]/(g(x))$ and
$\Z[x]/(g^*(x))$ must be non-isomorphic rings, as desired.
\end{proof}

The condition of being quasi-reduced
is fairly easy to attain:

\begin{lemma}\label{quasilemma}
If $g(x)$ is a monic real polynomial of nonzero discriminant, then
$g(\rho x)$ is quasi-reduced for any sufficiently large $\rho>0$.
\end{lemma}

\begin{proof}
This is easily seen from the Iwasawa-decomposition description of
Minkowski reduction. Consider the fundamental domain $\FF_\SL$ for
the action of $\SL_n(\Z)$ on $\SL_n(\R)$ given by
\begin{equation*}
\FF_\SL=\{\gamma=\nu \tau \kappa:\;\nu\in N'\;\tau\in T';\kappa\in\SO_n(\R)\},
\end{equation*}
where $N'$ denotes a compact subset (depending on $\tau$) of the group
of lower-triangular matrices and $T'$ is the group of diagonal
matrices $(t_1,\ldots,t_n)$ with $t_i\leq c\,t_{i+1}$ for all $i$ and
some absolute constant $c=c_n>0$.  Given an $n$-ary positive definite
integer-valued quadratic form $Q$, viewed as a symmetric $n\times n$
matrix, we write $Q=\gamma I_n \gamma^T$, where $I_n$ is the
sum-of-$n$-squares diagonal quadratic form and
$\gamma=\nu\tau\in\SL_n(\R)$ is unique up to right multiplication by
an element in $\SO_n(\R)$. The condition that $Q$ is Minkowski reduced
is equivalent to the condition that $\gamma$ belongs to $\FF_\SL$.
The condition that $Q$ be quasi-reduced is simply then that $t_i\leq
c\,t_{i+1}$ (with no condition on $\nu$).


Consider the natural isomorphism $\R[x]/(g(x))\to \R[x]/(g(\rho x))$
of \'etale $\R$-algebras defined by $x\to\rho x$.  If $\theta$ denotes
the image of $x$ in $\R[x]/(g(x))$, then $\rho\theta$ is the image of
$x$ in $\R[x]/(g(\rho x))$ under this isomorphism.
Let $Q_1$ be the Gram matrix of the lattice basis
$1,\theta,\theta^2,\ldots,\theta^{n-1}$ in
$\R^{n}$ and $Q_\rho$ be the Gram matrix of the lattice basis
$1,\rho\theta,\rho^2\theta^2,\ldots,\rho^{n-1}\theta^{n-1}$ in
$\R^{n}$.  If the element $\tau \in T$
corresponding to $g(x)$ is $(t_1,\ldots,t_n)$, then the element
$\tau_\rho\in T$ corresponding to $g(\rho x)$ is $(t_1,\rho t_2,\rho^2
t_3,\ldots,\rho^{n-1}t_n)$. This is because $Q_\rho=\Lambda
Q_1\Lambda^T$, where $\Lambda$ is the diagonal matrix
$(1,\rho,\rho^2,\ldots,\rho^{n-1})$; therefore, if
$Q_1=(\nu\tau\kappa)I_n(\nu\tau\kappa)^T$, then $$Q_\rho=
(\Lambda\nu\tau\kappa)I_n(\Lambda\nu\tau\kappa)^T=(\nu'(\Lambda\tau)\kappa)I_n(\nu'(\Lambda\tau)\kappa)^T$$
for some $\nu'\in N$ depending on $\Lambda$, so
$\tau_\rho=\Lambda\tau$. For sufficiently large $\rho$, we then have
$\rho^{i-1}t_i\leq c\rho^{i}t_{i+1}$ for all $i=1,\ldots,n-1$, as
desired.
 \end{proof}

 \noindent
 Lemma~\ref{quasilemma} implies that most monic irreducible integer
 polynomials are strongly quasi-reduced:

\begin{lemma}\label{mostqr}
A density of $100\%$ of irreducible monic integer polynomials
$f(x)=x^n+a_1x^{n-1}+\cdots+a_n$ of degree~$n$, when ordered by
height $H(f):=\max\{|a_1|,|a_2|^{1/2},\ldots,|a_n|^{1/n}\}$,
are strongly quasi-reduced.
\end{lemma}

\begin{proof}
Let $\epsilon>0$, and let $B$ be a compact region in $\R^n\cong
\Vmn(\R)$ consisting of monic real polynomials of nonzero discriminant
and height less than $1$ such that $$\Vol(B)>(1-\epsilon)\Vol(\{f\in
\Vmn(\R):H(f)<1\}).$$ For each $f\in B$, by Lemma~\ref{quasilemma}
there exists a minimal finite constant $\rho_f>0$ such that $f(\rho
x)$ is quasi-reduced for any $\rho>\rho_f$. The function $\rho_f$ is
continuous in $f$, and thus by the compactness of $B$ there exists a
finite constant $\rho_B>0$ such that $f(\rho x)$ is quasi-reduced for
any $f\in B$ and $\rho>\rho_B$.

Now consider the weighted homogeneously expanding region $\rho\cdot B$
in $\R^n\cong \Vmn(\R)$, where a real number $\rho>0$ acts on $f\in B$
by $(\rho\cdot f)(x)=f(\rho x)$.  Note that $H(\rho\cdot f)=\rho
H(f)$.  For $\rho>\rho_B$, we have that all polynomials in $\rho\cdot
B$ are quasi-reduced, and
$$\Vol(\rho\cdot B)>(1-\epsilon)\Vol(\{f\in \Vmn(\R):H(f)<\rho\}).$$
Letting $\rho$ tend to infinity shows that the density of monic
integer polynomials $f$ of degree $n$, when ordered by height, that
have nonzero discriminant and are strongly quasi-reduced is greater
than $1-\epsilon$. Since $\epsilon$ was arbitrary, and $100\%$ of integer
polynomials are irreducible, the lemma follows.
\end{proof}

We have the following variation of Theorem~\ref{polydisc2}.

\begin{theorem}\label{vanishinga1}
Let $n\geq1$ be an integer.  Then when
monic integer polynomials $f(x)=x^n+a_1x^{n-1}+\cdots+a_n$ of
degree~$n$ with $a_1=0$ are ordered by $H(f):=
\max\{|a_1|,|a_2|^{1/2},\ldots,|a_n|^{1/n}\}$, the density having
squarefree discriminant $\Delta(f)$ exists and is equal to
$\kappa_n=\prod_p\kappa_n(p)>0$, where $\kappa_n(p)$ is the density of
monic polynomials $f(x)$ over $\Z_p$ with vanishing
$x^{n-1}$-coefficient having discriminant indivisible by $p^2$.
\end{theorem}

Indeed, the proof of Theorem~\ref{polydisc2} applies also to those
monic integer polynomials having vanishing $x^{n-1}$-coefficient 
without any essential change; one simply replaces the representation
$W$ (along with $W_0$ and $W_{00}$) by the codimension-1 linear
subspace consisting of symmetric matrices with anti-trace $0$, but
otherwise the proof carries through in the identical manner.  The
analogue of Theorem~\ref{vanishinga1} holds also if the condition
$a_1=0$ is replaced by the condition $0\leq a_1<n$; in this case,
$\kappa_n=\prod_p\kappa_n(p)>0$ is replaced by the same constant
$\lambda_n=\prod_p\lambda_n(p)>0$ of Theorem~\ref{polydisc2}, since
for any monic degree-$n$ polynomial $f(x)$ there is a unique constant
$c\in\Z$ such that $f(x+c)$ has $x^{n-1}$-coefficient $a_1$ satisfying
$0\leq a_1<n$.

Lemmas \ref{reducedg} and \ref{mostqr} and Theorem~\ref{vanishinga1}
imply that 100\% of monic integer irreducible polynomials having
squarefree discriminant and vanishing $x^{n-1}$-coefficient (or those
having $x^{n-1}$-coefficient non-negative and less than $n$), when
ordered by height, yield {\it distinct} degree-$n$ fields.  Since
polynomials of height less than $X^{1/(n(n-1))}$ have absolute
discriminant $\ll X$, and since number fields of degree $n$ and
squarefree discriminant always have associated Galois group $S_n$, we
see that the number of $S_n$-number fields of degree $n$ and absolute
discriminant less than $X$ is $\gg
X^{(2+3+\cdots+n)/(n(n-1))}=X^{1/2+1/n}$. We have proven
Corollary~\ref{monogenic}.

\begin{remark}{\em
The statement of Corollary~\ref{monogenic} holds even if one specifies
the real signatures of the monogenic $S_n$-number fields of degree
$n$, with the identical proof.  It holds also if one imposes any
desired set of local conditions on the degree-$n$ number fields at a
finite set of primes, so long as these local conditions do not
contradict local monogeneity.  }\end{remark}

\begin{remark}{\em
We conjecture that a positive proportion of monic integer polynomials
of degree~$n$ with $x^{n-1}$-coefficient non-negative and less than
$n$ and absolute discriminant less than $X$ have height $O(
X^{1/(n(n-1))})$, where the implied $O$-constant depends only on $n$.
That is why we conjecture that the lower bound in
Corollary~\ref{monogenic} also gives the correct order of magnitude
for the upper bound.

In fact, let $C_n$ denote the $(n-1)$-dimensional Euclidean volume of
the $(n-1)$-dimensional region $R_0$ in $\Vmn(\R)\cong\R^n$
consisting of all polynomials $f(x)$ with vanishing
$x^{n-1}$-coefficient and absolute discriminant less than 1.  Then the
region $R_z$ in $\Vmn(\R)\cong\R^n$ of all polynomials $f(x)$ with
$x^{n-1}$-coefficient equal to $z$ and absolute discriminant less than
1 also has volume $C_n$, since $R_z$ is obtained from $R_0$ via the
volume-preserving transformation $x\mapsto x+z/n$.  Since we expect
that 100\% of monogenic number fields of degree $n$ can be expressed
as $\Z[\theta]$ in exactly one way (up to transformations of the form
$\theta\mapsto \pm \theta + c$ for $c\in \Z$), in view of
Theorem~\ref{polydiscmax2} we conjecture that the number of monogenic
number fields of degree $n$ and absolute discriminant less than $X$ is
asymptotic to
\begin{equation}
\frac{nC_n}{2\zeta(2)} X^{1/2+1/n}.
\end{equation}
When $n=3$, a Mathematica computation shows that we have
$C_3= \frac{2^{1/3}(3+\sqrt{3})}{45} \frac{\Gamma(1/2)\Gamma(1/6)}{\Gamma(2/3)}$}.
\end{remark}

Finally, we turn to the proof of
Corollary~\ref{shortvector}. Following \cite{EV}, for any algebraic
number $x$, we write $\|x\|$ for the maximum of the archimedean
absolute values of~$x$. Given a number field~$K$, write
$s(K)=\inf\{\|x\|: x\in\OO_K,\; \Q(x)=K\}$. We consider the number of
number fields $K$ of degree~$n$ such that $s(K)\leq Y$.

As already pointed out in \cite[Remark~3.3]{EV}, an upper bound of $\ll
Y^{(n-1)(n+2)/2}$ is easy to obtain.  Namely, a bound on the
archimedean absolute values of an algebraic number~$x$ gives a bound
on the archimedean absolute values of all the conjugates of $x$, which
then gives a bound on the coefficients of the minimal polynomial of
$x$. Counting the number of possible minimal polynomials satisfying
these coefficient bounds gives the desired upper bound.

To obtain a lower bound of $\gg Y^{(n-1)(n+2)/2}$, we use
Lemmas~\ref{reducedg} and \ref{mostqr} and
Theorem~\ref{vanishinga1}. Suppose $f(x)=x^n + a_2x^{n-2} + \cdots +
a_n$ is an irreducible monic integer polynomial of degree $n$.
Let~$\theta$ denote a root of $f(x)$.
If $H(f)\leq Y$, then $|\theta|\ll Y$;
this follows, e.g., from Fujiwara's bound~\cite{Fujiwara}:
$$ \|\theta\|\leq  \max\{ |a_1|,|a_2|^{1/2},\ldots,|a_{n-1}|^{1/(n-1)}|, |a_n/2|^{1/n}\}.$$

Therefore, if $H(f)\leq Y$, then
\begin{equation}\label{eq:est}
s(\Q[x]/(f(x))) \leq \|\theta\| \ll Y.
\end{equation}
Now Lemma~\ref{mostqr} and Theorem \ref{vanishinga1} imply that there are
$\gg Y^{(n-1)(n+2)/2}$ such polynomials $f(x)$ of height less than $Y$ that have squarefree discriminant and are also strongly
quasi-reduced. Lemma~\ref{reducedg} and \eqref{eq:est} then imply that
these polynomials define distinct $S_n$-number fields $K$ of degree $n$ with
$s(K)\leq Y$. This completes the proof of Corollary~\ref{shortvector}.

\subsection*{Acknowledgments}

We thank Levent Alpoge, Benedict Gross, Wei Ho, Kiran Kedlaya, Hendrik
Lenstra, Barry Mazur, Bjorn Poonen, Peter Sarnak, and Ila Varma for
their kind interest and many helpful conversations.  The first and third
authors were supported by a Simons Investigator Grant and NSF Grant DMS-1001828.

\end{document}